\documentclass[12pt,oneside,reqno]{amsart}
\RequirePackage[colorlinks,citecolor=blue,urlcolor=blue]{hyperref}
\usepackage{latexsym,amssymb,amsmath,amsfonts,amsthm}
\usepackage{amsthm,amsmath}
\usepackage{float}
\usepackage{enumerate}
\usepackage{fullpage}
\usepackage{bbm}
\usepackage{subfigure}
\usepackage{graphicx}
\usepackage[toc,page]{appendix}
\usepackage{amsfonts}
\usepackage[all]{xy}
\usepackage{caption}
\usepackage{booktabs}

\usepackage{geometry}
\usepackage{graphicx}
\usepackage{amssymb}
\usepackage{epstopdf}
\usepackage{verbatim}
\usepackage{color}
\usepackage{bbm, dsfont}
\usepackage{verbatim}
\usepackage[utf8]{inputenc}
\usepackage[dvipsnames]{xcolor}
\definecolor{myred}{RGB}{251,154,133}
\definecolor{myblue}{RGB}{153,206,227}
\definecolor{mylightblue}{RGB}{0, 150, 255}
\definecolor{mygreen}{RGB}{32, 210, 64}
\definecolor{mygray}{RGB}{220, 220, 220}

\usepackage{tikz}
\usetikzlibrary{decorations.pathmorphing}
\tikzset{snake it/.style={decorate, decoration=snake}}
\usetikzlibrary{shapes.geometric,positioning,decorations.pathreplacing}

\usepackage{cases}

\newtheorem{theorem}{Theorem}

\newtheorem{lemma}{Lemma}[section]
\newtheorem{corollary}{Corollary}[section]
\newtheorem{proposition}{Proposition}[section]

\theoremstyle{definition}

\theoremstyle{remark}
\newtheorem{remark}{Remark}

\definecolor{myred}{RGB}{251,154,133}
\definecolor{myblue}{RGB}{153,206,227}
\definecolor{mylightblue}{RGB}{0, 150, 255}
\definecolor{mygreen}{RGB}{32, 210, 64}
\definecolor{mygray}{RGB}{220, 220, 220}

\DeclareFontFamily{OML}{rsfs}{\skewchar\font'177}
\DeclareFontShape{OML}{rsfs}{m}{n}{ <5> <6> rsfs5 <7> <8> <9>
	rsfs7 <10> <10.95> <12> <14.4> <17.28> <20.74> <24.88> rsfs10 }{}
\DeclareMathAlphabet{\mathfs}{OML}{rsfs}{m}{n}

\newcommand{\BN}{{\mathbb{N}}}

\newcommand{\BP}{{\mathbb{P}}}

\newcommand{\BR}{{\mathbb{R}}}

\newcommand{\CB}{{\mathcal{B}}}

\newcommand{\CD}{{\mathcal{D}}}

\newcommand{\CF}{{\mathcal{F}}}

\newcommand{\CN}{{\mathcal{N}}}

\newcommand{\ind}{{\mathbbm{1}}}

\newcommand{\prob}{{\bf P}}

\newcommand{\e}{{\bf E}}
\newcommand{\bae}{\begin{equation}\begin{aligned}}
		\newcommand{\eae}{\end{aligned}\end{equation}}

\newcommand{\ep}{\varepsilon}
\newcommand{\beq}{\begin{equation}}
	\newcommand{\eeq}{\end{equation}}

\DeclareFontFamily{OML}{rsfs}{\skewchar\font'177}
\DeclareFontShape{OML}{rsfs}{m}{n}{ <5> <6> rsfs5 <7> <8> <9>
	rsfs7 <10> <10.95> <12> <14.4> <17.28> <20.74> <24.88> rsfs10 }{}
\DeclareMathAlphabet{\mathfs}{OML}{rsfs}{m}{n}

\begin{document}

	\title{Rigorous justification of the Fokker-Planck equations of neural networks based on an iteration perspective}
	
	\author{Jian-Guo Liu}
	\address[Jian-Guo Liu]{Department of Mathematics and Department of Physics, Duke University}
	%\urladdr{www.math.tamu.edu/~procaccia}
	\email{jliu@phy.duke.edu}

	\author{Ziheng Wang}
	\address[Ziheng Wang]{Mathematical Institute, University of Oxford}
	%\urladdr{http://www.math.tamu.edu/~tomye}
	\email{ wangz1@math.ox.ac.uk }
	
	\author{Yuan Zhang}
	\address[Yuan Zhang\footnote{NSFC Tianyuan Fund for Mathematics 12026606 }]
	{School of Mathematical Sciences, Peking University; Pazhou Laboratory, Guangzhou 510330, China}
	\email{zhangyuan@math.pku.edu.cn}

	\author{Zhennan Zhou}
	\address[Zhennan Zhou]
	{Beijing International Center for Mathematical Research, Peking University}
	\email{zhennan@bicmr.pku.edu.cn}
	
	\maketitle

\begin{abstract} In this work, the primary goal is to establish a rigorous connection between the Fokker-Planck equation of neural networks with its microscopic model: the diffusion-jump stochastic process that captures the mean field behavior of collections of neurons in the integrate-and-fire model. The proof is based on a novel iteration scheme: with an auxiliary random variable counting the firing events, both the density function of the stochastic process and the solution of the PDE problem admit series representations, and thus the difficulty in verifying the link between the density function and the PDE solution in each sub problem has been greatly mitigated. The iteration approach provides a generic frame in integrating the probability approach with PDE  techniques, with which we prove that the density function of the diffusion-jump stochastic process is indeed the classical solution of the  Fokker-Planck equation with a unique flux-shift structure.
%In fact, the iteration approach provides a generic frame in integrating the probability approach with PDE  techniques whenever the stochastic process possesses the renewal nature. As another application, we show that the diffusion-jump model with random discharges  converges to the original one in the sense of distribution  as the regularization parameter vanishes.
\end{abstract}

\maketitle

\section{Introduction}
While various models emerge in neuroscience \cite{gerstner2002spiking,lapique1907recherches,Nykamp_thesis, tuckwell1988introduction}, one of the most active disciplines at the present time, the level of mathematical rigor in understanding the rational connections between these models is usually formal or empirical. When it comes to modeling  the dynamics of a large collection of interacting neurons, the integrate-and-fire model for the potential through the neuron cell membrane, which dates back to \cite{lapique1907recherches}, has received great attention. In this model, the collective behavior of neuron networks can be predicted by the stochastic process of a single neuron \cite{brunel2000dynamics, brunel1999fast,compte2000synaptic,Delarue2015,guillamon2004introduction,Inglis2015,mattia2002population,newhall2010cascade,newhall2010dynamics,renart2004mean,sirovich2006dynamics,Touboul2012} where the influence from the network is given by an average synaptic input by the mean-field approximation \cite{Delarue2015,Inglis2015, renart2004mean,Touboul2012}. The time evolution of the probability density function (abbreviated by p.d.f.) of the potential voltage is governed by a Fokker-Planck equation on the half space  with an unusual structure: constantly shifting the boundary flux to an interior point. This equation has been utilized by neuroscientists to explore the macroscopic behavior of neural networks, and has also attracted many mathematicians to investigate the unique solutions structures in the past decade \cite{caceres2011analysis, caceres2011numerical,Caceres2014,Caceres2019, Caceres2018, carrillo2013classical, hu2021structure, omurtag2000simulation}, which in turn have enriched the scientific interpretation of the integrate-and-fire model. 

%In spite of the vast interests from the multiple research communities, the rigorous derivation of the Fokker-Planck equation from the microscopic stochastic model is not yet complete.

In this paper, we focus on the single neuron approximation of the celebrated noisy Leaky Integrate-and-Fire (LIF) model for neuron networks, where the state variable $X_t$ denotes the membrane potential of a typical neuron within the network. In the LIF model, when the synaptic input of the network (denoted by $I(t)$) vanishes, the membrane potential relaxes to their resting potential $V_L$, and in the single neuron approximation, the synaptic input $I(t)$, which itself is another stochastic process, is replaced by a continuous-in-time counterpart $I_c (t)$ (see  e.g. \cite{brunel2000dynamics, brunel1999fast,mattia2002population,omurtag2000simulation,renart2004mean,sirovich2006dynamics}), which takes the drift-diffusion form
\beq
I \,dt \approx I_c \, dt= \mu_c \, dt+\sigma_c \,dB_t.
\eeq
Here, $B_t$ is the standard Brownian motion, and in principle the two processes $I_c (t)$ and $I(t)$ have the same mean and variance. Thus between the firing events, the evolution of the membrane potential is given by the following stochastic differential equation
\beq
\label{process}
d X_t=(-X_t+V_L+\mu_c)\,dt+\sigma_c \,dB_t.
\eeq
The next key component of the model is the firing-and-reseting mechanism: whenever the membrane voltage $X_t$ reaches a threshold value called the threshold or firing voltage $V_F$, it is immediately relax to a reset value $V_R$, where $V_R < V_F$. The readers may refer to \cite{renart2004mean} for a thorough introduction of this subject. It is worth mentioning that, numerous mathematical aspects of the LIF model and its variants have been studied (see e.g. \cite{Delarue2015,delarue2015global,Inglis2015,MR3244555,renart2004mean,Touboul2012}) besides its enormous significance in neuroscience.

There has been a growing interest in studying the partial differential equation problem for the dynamics of the probability density function that the stochastic process $X_t$ is associated with \cite{chevallier2017mean, chevallier2015microscopic,Delarue2015,delarue2015global,Inglis2015}.  We denote the density of the distribution of neuron potential voltage at time $t \ge 0$ by $f(x,t)$, $x \in (- \infty, V_F ]$.  At least from a heuristic viewpoint,  it is widely accepted that the p.d.f. $f(x,t)$ satisfies the following Fokker-Planck equation on the half line with a singular source term
\beq
\frac{\partial f}{\partial t}(x,t)+\frac{\partial}{\partial x}[h f(x,t)]- a \frac{\partial^2 f}{\partial x^2} (x,t) = N(t)\delta (x-V_R),\quad x \in (- \infty, V_{F}), \quad t>0,
\label{master equation0}
\eeq
where $N(t)$ denotes the mean firing rate. By formal calculations via Ito's calculus, we obtain the drift velocity $h = -x + V_L + \mu_c$ and diffusion coefficient $a= \sigma_c^2/2$.

The firing-and-reset mechanism in the stochastic process has led to multiple consequences in the PDE model.  First, since the neurons at the threshold voltage has instantaneous discharges where the density is supposed to vanish and due to the noisy leaky terms, we consider the following Dirichlet boundary conditions:
\beq
f(V_F,t)=0, \quad f(-\infty, t)=0, \quad \forall t \ge 0.
\eeq

Second, due to the Dirichlet boundary condition at $x=V_F$, there is a time-dependent boundary flux escaping the domain, and a Dirac delta source term is added to the reset location $x=V_R$ to compensate the loss. Noting that \eqref{master equation0} is the evolution of a p.d.f, therefore for all $t\ge 0$
\[
\int_{-\infty}^{V_F} f(x,t) \, dx = \int_{-\infty}^{V_F} f_{\text{in}}(x) \, dx = 1.
\]
The conservation of mass and the boundary condition characterize the magnitude of mean firing rate
\beq
N(t) := - a \frac{\partial f}{\partial x} (V_F, t) \ge 0.
\eeq
The PDE problem is completed by an appropriate initial condition $f(x,0)= f_{\text{in}}(x)$. 

Third, the firing events generate currents that propagate within the neuron networks, which is incorporated into this PDE model by expressing the drift velocity $h$ and the diffusion coefficient  $a$ as functions of the mean-firing rate $N(t)$. For example, it is assumed in quite a few works (see e.g. \cite{caceres2011analysis, caceres2011numerical, carrillo2013classical, hu2021structure}) that
\[
h(x,N)=-x+ b N, \quad a(N)= a_0 + a_1 N,
\]
where $b$, $a_0 >0$ and $a_1 \ge 0$ are some modeling parameters. When $b>0$, the neuron network is excitatory on average, and when $b<0$ the network is inhibitory. In particular, when $b=0$ and $a_1=0$, the PDE problem becomes linear, but the flux shift structure persists.

We remark that this delta source term on the right hand side of \eqref{master equation0} is equivalent to setting the equation on $(-\infty, V_R) \cup (V_R, V_F)$ instead and imposing the following conditions
\[
f(V_R^-,t)=f(V_R^+,t), \quad  a\frac{\partial }{\partial x}f(V_R^-,t) - a\frac{\partial }{\partial x}f(V_R^+,t) = N(t), \quad  \forall t \ge 0.
\]
The equivalence can be checked by directly integration by parts and we choose to use this form for the rest of the paper.

Due to the unique structure of the PDE problem, most conventional analysis methods do not directly apply,  and many recent works are devoted to investigate  the solution properties of such model and its various modifications, including the finite-time blow-up of weak solutions, the multiplicity of the steady solutions, the relative entropy estimate, the existence of the classical solutions, the structure-preserving numerical approximation, etc.  (see e.g. \cite{caceres2011analysis,caceres2011numerical,Caceres2014,Caceres2018,carrillo2013classical, hu2021structure} and the references therein). For the stochastic process \eqref{process}, as the jumping time for $X_t$ is determined by its hitting time, the classical It\^{o} calculus is not directly applicable. 
%Thus the connection between the stochastic models and PDEs are more or less loose. In fact, to the best of our knowledge, whether the density of the stochastic process is a classical solution to the Fokker-Planck equation \eqref{master equation0} is not completely clear in literature.

%As a consequence, there exist quite a few similar but inconsistent results between the stochastic LIF model and the PDE model, such as the multiple firing event and the synchronous state \cite{caceres2011analysis,Caceres2014,Caceres2018}.

The primary goal of this paper is to show the rigorous derivation of the Fokker-Planck equation from the stochastic process. More specifically, we investigate whether and in which sense the probability density function $f(x,t)$ of the stochastic process $X_t$ satisfies the PDE model. In this paper, we choose the model parameters as follows
\beq \label{para_value}
V_L=V_R = 0, \quad \mu_c=0, \quad \sigma_c=\sqrt 2, \quad \text{and} \quad V_{F}=1.
\eeq
Let the distribution of $X_0$ be denoted by $\nu$, which is a probability measure compactly support on $(-\infty,1)$ and let $f_{\text{in}}(x)$ to denote the density function of $\nu$. Then $X_t \in ( -\infty, 1)$ is a stochastic process whose trajectory is $c\grave{a} dl\grave{a}g$ in time, and it evolves as an Ornstein–Uhlenbeck process
\beq
\label{eq:OU}
d X_t = - X_{t-} \,d t + \sqrt 2 \,d B_t,
\eeq
until it hits $1$. Whenever at time t, $X_t$ hits 1, it immediately jumps to $0$, i.e.
\beq  \label{eq:jumpc}
\text{if} \quad X_{t-}=1,  \quad X_t=0.
\eeq
Then we restart the O-U-like evolution independent of the past. We remark that \eqref{eq:OU} and \eqref{eq:jumpc} serve as a formal definition of the diffusion-jump process only for heuristic purposes and  the rigorous definition shall be presented in Section \ref{definition of process}. In this paper, we aim to show for any fixed $T>0$ that the associated density function $f(x,t)$ is indeed a classical solution to the PDE problem
\begin{equation}
	\label{eq_planck}
	\left\{
	\begin{aligned}
		&\frac{\partial f}{ \partial t} - \frac{\partial }{\partial x}\left( x f\right) - \frac{\partial ^2 f}{\partial x^2}=0, \quad x \in (-\infty,0) \cup (0,1), t\in (0,T],\\
		&f(0^-,t)=f(0^+,t), \quad  \frac{\partial}{\partial x}f(0^-,t)- \frac{\partial}{\partial x}f(0^+,t)=-\frac{\partial}{\partial x}f(1^-,t), \quad t\in (0,T],\\
		&f(-\infty,t)=0, \quad f(1,t)=0, \quad t\in [0,T],\\
		&f(x,0)=f_{\text{in}}(x), \quad x\in (-\infty,1).
	\end{aligned}
	\right.
\end{equation}

The processes of such type \eqref{eq:OU} and \eqref{eq:jumpc}  were first introduced by Feller \cite{MR47886, MR63607} (in terms of transition semigroups). In particular, \cite{MR63607} presents the Fokker-Planck equation of such processes (dubbed ``elementary return process" there) in a weak form, of which the proof is based on a Markov semigroup argument in \cite{MR47886}. See Theorem 9 of \cite{MR63607} for details.
Such processes have also been studied in later works such as \cite{MR2499861, MR0346904, MR3244555, MR3016590, MR573203, MR2673971}. More specifically, in \cite{MR2353702,MR2499861, MR3244555, MR3016590}, the authors are concerned with the spectral properties of the generator of the stochastic process or related models, and have shown the exponential convergences in time towards the stationary distribution. In particular, \cite{MR3244555} applied their results to a neuronal firing model driven by a Wiener process and computed the distribution of the first passage time. In the works \cite{MR573203,MR2673971}, the authors made more relaxed or modified assumptions on the stochastic process than those in \cite{MR0346904}, and proved the existence of pathwise solution of such process in a generalized sense.

Following the spirit of the pioneering work of Feller \cite{MR63607}, the focus of this paper is to rigorously establish the bridge between the density functions of such processes and the classical solutions of the Fokker-Planck equations to be specified as in \eqref{eq_planck}. From the technical perspective, there is no available mathematical tools to link the boundary condition at the firing  voltage and the jump condition at the reset voltage  (or equivalently, the singular delta source term) of the PDE model to stochastic model for a single neuron model. In \cite{caceres2011analysis,carrillo2013classical}, some heuristic arguments are provided to connect $N(t)$ to the rate of change of the expectation of the number of firing events, which is related to the synchronization behavior of the neuron networks. Whereas, such an interpretation is not applicable for a single neuron model. In this paper, we rigorously prove that for a single neuron the mean firing rate $N(t) = \sum_{n=1}^\infty f_{T_n}(t)$ where $f_{T_n}$ stands for the p.d.f. of the n-th jumping time of $X_t$.

The key strategy of our proofs in this paper is based on an iterated scheme: with the introduction of an auxiliary random variable counting the number of firing events, the probability density function of potential voltage  $f(x,t)$  allows a decomposition  as a summation of sub-density functions $\{f_n (x,t) \}_{n=0}^{\infty}$. Each sub-density naturally links to a less singular sub-PDE problem, and all the sub-PDE problems are connected successively by iteration: the escaping boundary flux of $f_n(x,t)$ serves as the singular source for $f_{n+1}(x,t)$. Among all the iterations, the first step from $f_0$ to $f_1$ exhibits strongest singularity at the source of the flux, and thus turns out to be the major technical difficulty in our proof. In order to tackle this obstacle, elaborated estimates on the regularities of $f_0$ have to be established. The first sub-PDE problem corresponds to the stochastic process killed at the first hitting time and there is a vast literature \cite{2019Global, 2017A, 2018A, 2017Particle, 2020Mean} concerned with the stochastic processes with no reset for the killed particles. In \cite{delarue2013first,delarue2015global}, the authors consider the process with firing-and-reseting as in this paper and have established the connections between the sub-density function and the PDE solution. They have proved that $f_0(t, x)$ is continuous in $(t, x)$ and continuously differentiable in x on $(0,T] \times (-\infty, 1]$ and admits Sobolev derivatives of order $1$ in $t$ and of order $2$ in $x$ on any compact subset of $(0,T] \times (-\infty, 1)$. However, these results are not strong enough to guarantee the existence of the classical solution to the whole problem (9). In fact, by analysing the Green's function for the parabolic equation on the half space, we get estimates for classical derivatives and high order regularity for $t$ in Proposition 3.1 and 3.2, which is essential for the iteration from $f_0$ to $f_1$. Besides, all the desired smoothness properties are maintained by the iteration scheme, and thanks to the decomposition, rigorous justification of the jump condition for each sub-PDE problem becomes tractable. Finally, with the exponential convergence of decomposition, we can pass to the limit, and  conclude the preserved properties on the original problem. This iteration scheme is inspired by the renewal nature of the stochastic process, which shares the spirit of Feller's original work in \cite{MR63607}, and provides a platform to combine the techniques from both the probability theory and the differential equations.  

It is worth noting that, as the first attempt to study the rigorous justification of the Fokker-Planck equations of neural networks from the stochastic model, we have only obtained the results for the linear cases. In particular, we could not incorporate the dependence on the mean firing rate in the drift velocity and in the diffusion coefficient yet, but we shall investigate those directions in the future.

The rest of the paper is outlined as follows. In Section \ref{sec:pre}, we summarize the main results of this work as well as give precise definition of the stochastic process and lay out the iterated scheme. In Section \ref{iteration}, we show that the density function of the stochastic process is indeed the mild solution of the PDE problem with certain smoothing properties, and we give a few remarks on the implications in the weak solution. For the rest of this work, we use $C$, $C_0$, $C_k$ and $C_T$ to denote generic constants.

\section{Preliminaries and Main Results} \label{sec:pre}

In this section, we present the main results of this paper in details, and also provide some technical preparations for the proofs, including the construction of the stochastic process, which serves as the precise definition, and the elaboration of the iterated strategy, accompanied by some elementary estimates.

\subsection{Main Results}\label{main result}
\

The stochastic process $X_t$ has been formally defined in \eqref{eq:OU} and \eqref{eq:jumpc}, but note that the rigorous  construction of such a process can be found in \eqref{process definition} below of Section \ref{definition of process}.  
We first suppose that the process $X_t$ starts from $0$, i.e. the distribution of $X_0$ is $f_{\text{in}}(x)=\delta(x)$. We state the first main result in the following
\begin{theorem}
\label{thm_forward}
The process $X_t$ as in \eqref{process definition} that starts from $0$ has a continuously evolved probability density function denoted by $f(x,t)$. $f(x,t)$ is a solution of \eqref{eq_planck} in the time interval $(0, T]$ for any given $0<T< +\infty$ and with initial condition $\delta(x)$ in the following sense:
\begin{itemize}
    \item[(\romannumeral1)]
     $N(t):=-\frac{\partial}{\partial x}f(1^-,t)$ is a continuous function for $t\in [0, T]$,
    \item[(\romannumeral2)]
    $f$ is continuous in the region $\{(x,t):-\infty< x \le 1, \ t\in (0, T] \}$,
    \item[(\romannumeral3)]
    $f_{xx}$ and $f_t$ are continuous in the region $\{(x,t): x\in (-\infty,0)\cup(0,1), \ t\in (0,T] \}$,
    \item[(\romannumeral4)]
    $f_x(0^-, t)$, $f_x(0^+, t)$ are well defined for $t\in (0, T]$
    \item[(\romannumeral5)]
    For $t\in (0, T]$, $f_x(x, t)\to 0$ when $x \to -\infty$,
    \item[(\romannumeral6)]
    Equations \eqref{eq_planck} are satisfied with $f(x,0)=\delta(x)$ in the following sense: for any $\varphi\in C_b(-\infty,1)$,
    \beq
    \lim_{t\to 0^+}\int_{-\infty}^1  \varphi(x) f(x,t)dx=\varphi(0).
    \eeq
\end{itemize}
\end{theorem}
The proof of Theorem \ref{thm_forward} is shown in Section \ref{iteration}, which relies on an iteration approach. In fact, we decompose both the probability density of the stochastic process and the solution to equation \eqref{eq_planck} into series and show that there is a one-to-one correspondence between the two series representations.

Next we let the process start from any fixed $y<1$, this time we use $f^y(x,t)$ to denote the p.d.f the process $X_t$ in \eqref{process definition} start from $y$ and now the distribution of $X_0$ is $f_{\text{in}}(x)=\delta(x-y)$. With the same method, we get the following corollary immediately.
\begin{corollary}
\label{thm_forward y}
For any fixed $y\in(-\infty, 1)$, the process $X_t$ as in \eqref{process definition} that starts from $y$ has a continuously evolved probability density function denoted by $f^y(x,t)$. $f^y(x,t)$ is a solution of \eqref{eq_planck} in the time interval $(0, T]$ for any given $0<T< +\infty$ and with initial condition $\delta(x-y)$ in the following sense:
\begin{itemize}
    \item[(\romannumeral1)]
    $N^y(t)$ is a continuous function for $t\in [0, T]$,
    \item[(\romannumeral2)]
    $f^y$ is continuous in the region $\{(x,t):-\infty< x \le 1, \ t\in (0, T] \}$,
    \item[(\romannumeral3)]
    $\partial_{xx}f^y$ and $\partial_tf^y$ are continuous in the region $\{(x,t): x\in (-\infty,0)\cup(0,1), \ t\in (0,T] \}$,
    \item[(\romannumeral4)]
    $\partial_xf^y(0^-, t)$, $\partial_xf^y(0^+, t)$ are well defined for $t\in (0, T]$
    \item[(\romannumeral5)]
    For $t\in (0, T]$, $\partial_xf^y(x, t)\to 0$ when $x \to -\infty$,
    \item[(\romannumeral6)]
    Equations \eqref{eq_planck} are satisfied  with $f(x,0)=\delta(x-y)$ in the following sense: for any $\varphi\in C_b(-\infty,1)$,
    \beq
    \lim_{t\to 0^+}\int_{-\infty}^1  \varphi(x) f^y(x,t)dx=\varphi(y).
    \eeq
\end{itemize}
Moreover, for any fixed $\varepsilon_0>0$, the continuity in (\romannumeral1), (\romannumeral2), (\romannumeral3) and the convergence in (\romannumeral5) and (\romannumeral6) are uniform for $y\le1-\varepsilon_0$.
\end{corollary}
The proof of Corollary \ref{thm_forward y} is exactly same as in Theorem \ref{thm_forward}, and is thus skipped.

The initial condition of the Fokker-Planck equation \eqref{eq_planck} corresponds to the initial distribution of the stochastic process $X_0$. We remark that, in the above cases the most arguments below are based on the initial condition of the process $X_0=y$ for any $y<1$, and the corresponding initial condition of the PDE problem becomes $f(x,0)=\delta(x-y)$. Although the initial condition is a singular function, we have shown that PDE has an instantaneous smoothing effect, while the solution coincide with the density function of the stochastic process. Since the problem is linear, the natural extension to general and proper initial conditions can be obtained by integration against the initial distribution (see e.g. \cite{delarue2013first} for a careful discussion).

\begin{theorem}
\label{thm_forward nu}
Let $\nu$ be a c.d.f. whose p.d.f. $f_{\text in}(x)\in C_c(-\infty,1)$. We assume that $f_{\emph{in}}(x)$ is continuous and supported in $(-\infty, 1-\varepsilon_0)$ for some $\varepsilon_0>0$. Then the process $X_t$ as in \eqref{process definition} that starts from p.d.f. $f_{\text{in}}(x)$ has a continuously evolved probability density function denoted by $f^\nu(x,t)$ with
\beq
\label{convolution}
f^\nu(x,t)=\int_{-\infty}^{1-\varepsilon_0}f^y(x,t)\nu(dy),\quad x\in(-\infty,1],\quad t>0,
\eeq
and $f^\nu(x,t)$ is a classical solution of \eqref{eq_planck} in the time interval $(0, T]$ for any given $0<T< +\infty$ with initial condition $f_{\emph{in}}(x)$ in the following sense:
\begin{itemize}
    \item[(\romannumeral1)]
    $N^\nu(t):=-\frac{\partial}{\partial x}f^\nu(1^-,t)$is a continuous function for $t\in [0, T]$,
    \item[(\romannumeral2)]
    $f^\nu$ is continuous in the region $\{(x,t):-\infty< x \le 1, \ t\in [0, T] \}$,
    \item[(\romannumeral3)]
    $\partial_{xx}f^\nu$ and $\partial_xf^\nu$ are continuous in the region $\{(x,t): x\in (-\infty,0)\cup(0,1), \ t\in [0,T] \}$,
    \item[(\romannumeral4)]
    $\partial_xf^\nu(0^-, t)$, $\partial_xf^\nu(0^+, t)$ are well defined for $t\in [0, T]$,
    \item[(\romannumeral5)]
    For $t\in (0, T]$, $\partial_xf^\nu(x, t)\to 0$ when $x \to -\infty$.
    \item[(\romannumeral6)]
    Equations \eqref{eq_planck} are satisfied with the $L^2$ convergence to the initial condition as $t\to 0^+$, i.e.
    \beq
    \label{L2 convergence}
    \lim_{t\to 0^+} \int_{-\infty}^{1}|f^\nu(x,t)-f_{\text{in}}(x)|^2dx=0.
    \eeq.
\end{itemize}
\end{theorem}
A proof can be found at the end of Section \ref{solutions}.
\begin{remark}
		It is not clear yet how to get the uniform estimates near the boundary of the domain and thus we suppose that the initial distribution compactly support on $(-\infty,1)$ in this paper. Actually, some recent works \cite{2018A} concerning related models progressed towards more general assumptions, from compactly supported to $o(1-x)$ decay near 1 and more recently $\mathcal{O}\left((1-x)^{\beta}\right)$ with $\beta \in(0,1)$.  Usually, the literature assumes $\mathcal{O}(1-x)$ decay near 1 (e.g. \cite{2015Qualitative}) and in Theorem 1.1 of  \cite{2019Global}, this boundary decay is linked with short-term regularity of the solutions. Thus the hypothesis of a compactly supported initial condition has deep consequences upon the smoothness of the solution in short time.
\end{remark}

\subsection{Construction of the Process}\label{definition of process}
\

%\noindent

For the rest of this section, we shall present some preliminaries of  the stochastic process. Firstly, we should give the process $X_t$ a precise definition in probability by following the construction by Gihman and Skorohod \cite{MR0346904}. We emphasize that, an addition process $n_t$ is introduced to count the number of jumping events of a trajectory that has taken place before time $t$.

On a given probability space $(\Omega, \CF, \prob )$, we consider a sequence of independent O-U processes
$$
\left\{Y_t^{(n)}\right\}_{n=1}^\infty
$$
with $Y_0^{(n)}=0$ for all $n\ge 1$. Note that an O-U process $Y_t$ starting from initial value $y_0$ is an SDE with a.s. pathwisely continuous strong solution. That is
\begin{equation}
\label{OU_sol}
Y_t=e^{-t}y_0+\sqrt 2 \int^t_0 e^{-(t-s)}dB_s
\end{equation}
with a normal p.d.f.:
\begin{equation}
N(e^{-t}y_0, 1-e^{-2t}).
\label{normal}
\end{equation}
For each $n\in \BN$, $t\in [0,\infty]$, define the natural filtration
$$
\CF^{(n)}_t=\sigma\left( Y_s^{(n)}: s\in [0,t)\right).
$$
I.e., $\CF^{(n)}_t$ represents the information carried by the path of the $n$th copy of O-U process by time $t$.  For all $n$, $\CF^{(n)}_\infty$ are abbreviated to $\CF^{(n)}$, which are easy to see to be jointly independent.  Now define their filtration
$$
\mathcal{G}_n=\sigma(\CF^{(k)}, \ k\le n), \ \mathcal{G}^n=\sigma(\CF^{(k)}, \ k\ge n)
$$
with the convention $\mathcal{G}_\infty=\mathcal{G}$.

For each $n$, let
\beq
\label{stopping time sequence}
\tau_n=\inf\left\{t\ge 0: Y_t^{(n)}=1 \right\}=\inf\left\{t\ge 0: \lim_{h\to t^-}Y_h^{(n)}=1 \right\}
\eeq
be the first time $Y_t^{(n)}$ hits 1, with the convention $\tau_0=0$. Moreover, for all $n\ge 0$ and $k\le n$, define
\beq
\label{stopping time}
T_n=\sum_{i=0}^n \tau_i,\quad T_{n,k}=\sum_{i=k+1}^n \tau_i.
\eeq

By definition, $\tau_n$ is a stopping time with respect to the natural filtration $\{\CF^{(n)}_t\}_{t\ge 0}$. And we have that $\{\tau_n\}_{n=1}^\infty$ is a sequence of i.i.d. r.v.'s with strictly positive expectation. Thus by the law of large numbers, $(\sum_{i=1}^n \tau_i)/n\to E[\tau_1]>0$ a.s., which implies that
$$
\prob\left(\sum_{i=k}^\infty \tau_i=\infty, \ \forall k\ge 1 \right)=1.
$$
Particularly, we have $T_n\to\infty$ a.s. as $n\to\infty$. Then within the almost sure event $A_0=\{\sum_{i=k}^\infty \tau_i=\infty, \ \forall k\ge 1\}$, we define $(X_t,n_t)$ as follows: for any $k\ge 1$
\beq
\label{process definition}
(X_t, n_t)= \big(Y^{(k)}_{t-T_{k-1}}, k-1\big)
\eeq
on $[T_{k-1},T_{k})$. And thus $T_k$ is interpreted as the $k$-th jumping time associated with $X_t$.

By definition, we have constructed a piecewise continuous path on $[0,\infty)$ for each $\omega\in A_0$, and thus a mapping from $A_0$ to $(D[0,\infty)\times \BN, \CD\times \CN)$ is clearly measurable with respect to $\mathcal{G}$, where $D[0,\infty)$ is the space of $c\grave{a} dl\grave{a}g$ paths. Here $\CD$ is the smallest sigma field generated by all coordinate mappings and $\CN$ is the trivial sigma field on $\BN$. In the rest of this paper, we will use the construction above as the formal definition of $(X_t,n_t)$, which is the stochastic process of interest.

Similarly, we can define the process $X_t$ starts from $y<1$ or starts from a distribution $\nu$. We denote the  probability measure of $(X_t,n_t)$ by $\prob^y(\cdot)$ and the expectation by $\e^y[\cdot]$. The meaning of $\prob^\nu(\cdot)$ and $\e^\nu[\cdot]$ are analogous. Using $F_{\tau_k}$/$F_{T_k}$ to denote the cumulative distribution function of $\tau_k$/$T_k$, it is immediate to see that for any $k$ and $t$, $\prob(\tau_k=t)\le \prob(Y^{(k)}_t=1)=0$. So $F_{\tau_k}$ and $F_{T_k}$ are always continuous.

\subsection{Properties of the Process and the Iteration Approach}
\

%\noindent

We derive some preliminary estimates for the process $(X_t,n_t)$, which manifest the solution properties and also motivate us to propose the iterated scheme.

It has been shown in \cite{MR0346904} the process $X_t$ constructed above is always Markovian. Now we are ready to show the following ``Strong Markovian" type result that allows us to later calculate the probability distribution of $(X_t, n_t)$ in an iterative fashion: for each integer $k\ge 0$, define
\beq
\label{iteration_cdf}
F_k(x,t)=\prob^0(X_t\le x, n_t=k),
\eeq
then we have
\begin{proposition}
\label{proposition_strong}
For any $x<1$, $k\ge 1$, and $t>0$,
\begin{equation}
\label{eq_strong}
F_k(x,t)=\e_{0}\left[\prob\left(Y^{(k+1)}_{t-T_k}\le x, \tau_{k+1}>t-T_k\right)\mathbbm{1}_{T_k<t} \right].
\end{equation}
And thus
\beq
\label{k-iteration}
F_k(x,t)=\int_0^t F_0(x,t-s)dF_{T_k}(s).
\eeq
\end{proposition}
\begin{proof}
We only prove \eqref{eq_strong} and then \eqref{k-iteration} is obvious. First, note that $T_{k+1}=T_k+\tau_{k+1}$ and that the event
$$
\{n_t=k\}=\{T_k\le t, T_{k+1}>t\}.
$$
By Fubini's formula,
$$
\prob^0(n_t=k)=\e_{0}\left[ \prob\left(\tau_{k+1}>t-T_k\right)\mathbbm{1}_{T_k<t} \right].
$$
Thus it suffices to prove
$$
\prob^{0}(X_t> x, n_t=k)=\e_{0}\left[ \prob\left(Y^{(k+1)}_{t-T_k}> x, \tau_{k+1}>t-T_k\right)\mathbbm{1}_{T_k<t} \right].
$$
Let $A=\{X_t>x, n_t=k\}$ be our event of interest. For any $n\ge 1$ and any $0\le i\le 2^n-1$, we define interval
$$
I^{(i)}_n(t)=\left(2^{-n}it, 2^{-n}(i+1)t \right].
$$
Moreover, for any $s\in (0,t]$ and any $n$, one may define $\text{Id}(n,s)$ be the unique $i\le 2^n-1$ such that $s\in I^{(i)}_n(t)$. Now we define event
$$
A^{(i)}_n=\left\{\inf_{s\in t-I^{(i)}_n(t)} Y^{(k+1)}_s>x, \ \tau_{k+1}>(1-2^{-n}i)t \right\}\cap \{T_k\in I^{(i)}_n(t)\}
$$
and $A_n=\cup_{i=0}^{2^n-1} A^{(i)}_n$. By definition, $A^{(i)}_n\subset A$ for every feasible $n$ and $i$. Thus $\prob(A_n)\le \prob(A)$. On the other hand, for any $\omega\in \bar A=\{X_t>x, n_t=k, T_k<t\}$, The continuity of path in $Y^{(k+1)}$ guarantees that there has to be some $N<\infty$ such that for all $n\ge N$, $\omega \in A^{(\text{Id}(n,T_k(\omega)))}_n$ and thus $\prob^{0}(A_n)\to \prob^{0}(\bar A)=\prob^{0}(A)$ as $n\to \infty$. The last equality follows from the fact that $F_{T_k}$ is continuous.

Meanwhile, note that $T_k$ is independent to $Y^{(k+1)}$. We have
\begin{align*}
\prob^{0}(A_n)&=\sum_{i=0}^{2^n-1} \prob^{0}\left(T_k\in I^{(i)}_n(t)\right) \prob\left(\inf_{s\in t-I^{(i)}_n(t)} Y^{(k+1)}_s>x, \ \tau_{k+1}>(1-2^{-n}i)t  \right)\\
&=\e^{0}\left[\prob\left(\inf_{s\in t-I^{(\text{Id}(n, T_k))}_n(t)} Y^{(k+1)}_s>x, \ \tau_{k+1}>(1-2^{-n} \text{Id}(n, T_k))t  \right) \mathbbm{1}_{T_k<t} \right].
\end{align*}
Now noting that for any $0<h<t$, one may similarly have from the continuity of $Y^{(k+1)}$
$$
\prob\left(\inf_{s\in t-I^{(\text{Id}(n, T_k))}_n(t)} Y^{(k+1)}_s>x, \ \tau_{k+1}>(1-2^{-n} \text{Id}(n, h))t  \right) \to \prob\left(Y^{(k+1)}_{t-h}>x, \tau_{k+1}>t-h\right),
$$
we have \eqref{eq_strong} follows from monotone convergence.

\end{proof}

For any $t>0$, we first consider the case where no jumps have been made by time $t$. Note that $F_0(x,t)=P(X_t\le x, T_1>t)=\prob(Y^{(1)}_t\le x, \tau_1>t)$ for all $x\in (-\infty, 1)$. It is clear that $F_0(\cdot, t)$ induces a measure on $( (-\infty, 1), \CB)$, which is absolutely continuous with respect to the Lebesgue measure on $\BR$. The assertion above can be seen from the fact that for any measurable $A$, $\prob(Y^{(1)}_t\in A, \tau_1>t)\le \prob(Y^{(1)}_t\in A)$ and that $Y^{(1)}_t$ is a continuous random variable. Here we also use $F_0(\cdot,t)$  to denote the corresponding measure on $( (-\infty, 1), \CB)$. And let $f_0(x,t)$ be its density and $p_{\text{ou}}(x,t)$ denotes the p.d.f of $Y^{(1)}_t$. Thus we have
\beq
\label{decay at infinity}
f_0(x,t)\le p_{\text{ou}}(x,t)=\frac{1}{\sqrt{2\pi(1-e^{-2t})}}\exp{\{\frac{-x^2}{2(1-e^{-2t})}\}},
\eeq
which together with \eqref{normal} derives
\begin{equation}
\label{density}
f_{0}(x,t)\le \frac{1}{\sqrt{2\pi(1-e^{-2t})}}.
\end{equation}
%\begin{remark}
%\eqref{density} holds because of (2.214) of page 69 of \cite{platen2010numerical}.
%\end{remark}
%}

\begin{lemma}
\label{CDF_continuity}
$F_0(x, t)$ is a bi-variate continuous function on $(-\infty,1]\times(0,\infty)$. Moreover, for any bounded continuous function $\varphi(x)$,
\beq
\label{eq_initial_delta}
\lim_{t\to 0^+}\e^0[\varphi(X_t)\ind_{n_t=0}]=\lim_{t\to 0^+}\int_{-\infty}^1  \varphi(x) f_0(x,t)dx=\varphi(0).
\eeq
\end{lemma}

\begin{proof}
In order to prove this lemma, one may first show that for any $(x,t)\in (-\infty,1)\times(0,\infty)$, $F_0(\cdot,\cdot)$ is continuous at $(x,t)$ on both directions.

The continuity on the direction of $x$ is obvious since that for all $x'>x$,
$$
0\le F_0(x', t)-F_0(x, t)\le \prob\left(Y^{(1)}_t\in [x,x']\right)=\int_x^{x'} p_{\text{ou}}(y,t) dy
$$		
and the last term goes to $0$ as $x' \to x^+$.

Thus one may concentrate on proving continuity on the direction of $t$. Let $\Delta$ be the symmetric difference between events. One may first note that for any events $A=A_1\cap A_2$, and $B= B_1\cap B_2$,
\bae
\label{eq_symmetric}
\begin{aligned}
A\Delta B&=(A_1\cap A_2\cap B_1^c)\cup (A_1\cap A_2\cap B_2^c)\cup (A_1^c\cap B_1\cap B_2)\cup (A_2^c\cap B_1\cap B_2)\\
&\subset (A_1\cap B_1^c)\cup (A_2\cap B_2^c)\cup (A_1^c\cap B_1)\cup (A_2^c\cap B_2)\\
&=(A_1\Delta B_1)\cup (A_2\Delta B_2).
\end{aligned}
\eae
For any $t>0$, fixed $x_0$ and any $\Delta t$ sufficiently close to 0 (without loss of generality, one may assume $\Delta t>0$)
$$
\begin{aligned}
&F_0(x_0, t)=\prob(Y^{(1)}_t\le x_0, \ \tau_1>t)\\
&F_0(x_0, t+\Delta t)=\prob(Y^{(1)}_{t+\Delta t}\le x_0, \ \tau_1>t+\Delta t).
\end{aligned}
$$
Now let $A_1=\{Y^{(1)}_t\le x_0\}$, $A_2=\{\tau_1>t\}$, and $B_1=\{Y^{(1)}_{t+\Delta t}\le x_0\}$, $B_2=\{\tau_1>t+\Delta t\}$. By \eqref{eq_symmetric} we have
\begin{equation*}
\begin{aligned}
&\left|F_0(x_0, t)-F_0(x_0, t+\Delta t)\right|\\
\le &\prob(A\Delta B)\le \prob(A_1\Delta B_1)+\prob(A_2\Delta B_2)\\
=&\prob\left(Y^{(1)}_t\le x_0, Y^{(1)}_{t+\Delta t}> x_0\right)+\prob\left(Y^{(1)}_t> x_0, Y^{(1)}_{t+\Delta t}\le x_0\right)+\prob(\tau_1\in (t,t+\Delta t])\\
\le &\prob\left(\exists s\in [t, t+\Delta t], \ s.t. \ Y^{(1)}_s=x_0\right)+F_{\tau_1}(t+\Delta t)-F_{\tau_1}(t).
\end{aligned}
\end{equation*}
Recalling that $F_{\tau_1}(\cdot)$ is continuous,
$$
\lim_{\Delta t\to 0} F_{\tau_1}(t+\Delta t)-F_{\tau_1}(t)=0.
$$
At the same time, for any positive integer $n$, define event
$$
\Delta_n=\left\{\exists s\in \big[t, t+n^{-1}\big], \ s.t. \ Y^{(1)}_s=x_0\right\}.
$$
Note that
$$
\prob(\Delta_n)\to \prob(Y^{(1)}_t=x_0)=0 \quad as \quad n\to \infty.
$$
We can got the continuity of $t$.

Thus, one can show that $F_0(x,t)$ is binary continuous at $(x,t)$ as follows:
given $(x,t) \in (-\infty, 1)\times(0,+\infty)$ and any $\epsilon > 0, \exists 0<\delta<\frac {t}{2}$ such that for any $|t'-t|\le \delta$,
$$
|F_0(x, t')-F_0(x, t)|< \frac{\epsilon}{2}.
$$
And for any $s>\frac{t}{2}$ and any $|x'-x|\le \delta$ (Here without loss of generality, we ask $x<x'$)
$$
|F_0(x', s)-F_0(x, s)|\le \prob(Y_s^{(1)} \in [x, x'])<\frac{\epsilon}{2}.
$$
(The last inequality is because when $s<\frac t2$, the density of $Y_s^{(1)}$ can be bounded by a big enough constant $C$.)\\
Then for all $(x', t')\in (-\infty, 1)\times(0, \infty)$ such that $|t'-t|\le \delta, |x'-x|\le \delta$, we have
$$
|F_0(x', t')-F_0(x, t)|\le |F_0(x', t')-F_0(x, t')|+|F_0(x, t')-F_0(x, t)|<\epsilon
$$

Finally, we show that $F_0(x,t)$ is continuous at $x=1$. It suffice to prove that for any $t_n \to t$ and $\varepsilon_n \to 0^+$, we have $\lim_{n \to \infty} F_0(1-\varepsilon_n, t_n)=F_0(1,t)=\prob(\tau_1>t)$, i.e. $\lim_{n\to \infty}\prob(X_{t_n}\le 1-\varepsilon_n, \tau_1>t_n)=\prob(\tau_1>t)$, which is equivalent to
$$
\lim_{n\to \infty}\prob(X_{t_n}> 1-\varepsilon_n, \tau_1>t_n)=0.
$$
Set event $A_n=\{X_{t_n}>1-\varepsilon_n, \tau_1>t_n \}$ and we have
$$
\prob(\cup_{m \ge n}^{\infty} A_m)\le \prob(\exists s \in[\max_{m\ge n} t_m, \max_{m\ge n} t_m], \ s.t.\  X_s>1-\varepsilon_n, \tau_1>\min_{m\ge n} t_m).
$$
Note that $\limsup_{n\to \infty} \prob(A_n)\le \prob(\limsup A_n)\le \prob(X_t\ge 1, \tau_1\ge t)=0$. Thus we get $\lim_{n \to \infty} \prob(A_n)=0$ and get the result we want.

Finally to prove \eqref{eq_initial_delta}, recall that $\varphi$ is a bounded and continuous function. Thus $|\varphi(x)|\le M$ for all $x$, and for each $\ep>0$, there is a $0<\delta<1$ such that for all $x\in [-\delta,\delta]$, $|\varphi(x)-\varphi(0)|<\ep$. So we have
$$
\left|\e_0[\varphi(X_t)\ind_{n_t=0}]-\varphi(0)\right|\le \ep+ 2M \prob\left(\max_{s\le t}|Y^{(1)}_s|\ge \delta\right).
$$
Now recalling \eqref{OU_sol}
\beq
\label{eq_max_Y1}
|Y^{(1)}_t|\overset{d}{=}\left|\sqrt 2 \int^t_0 e^{-(t-s)}dB_s\right|\overset{d}{\le} \left|\sqrt 2 \int^t_0 e^{s}dB_s\right|,
\eeq
where the $d$ means the probability distribution. Note that the right hand side of \eqref{eq_max_Y1} forms a martingale. One immediately have
$$
\lim_{t\to 0^+}\prob\left(\max_{s\le t}|Y^{(1)}_s|\ge \delta\right)=0
$$
by Doob's inequality. Thus we have shown \eqref{eq_initial_delta} and then completed the proof.

\end{proof}

\begin{remark}
With Lemma \ref{CDF_continuity}, one may immediately have that $F(x_0, t)$ is a bounded and measurable function of $t\in [0,\infty)$.
\end{remark}

Moreover, the following corollary follows directly from Proposition \ref{proposition_strong}, Lemma \ref{CDF_continuity}, and a standard measure theory argument:

\begin{corollary}
\label{col_function_iteration_1}
For any bounded measurable function $f$, any integer $k\ge 1$ and any $t>0$, $\e\left[f(Y^{(1)}_t)\ind_{\tau_1>t}\right]$ is measurable with respect to $t$, and
\beq
\e^0\left[ f(X_t) \ind_{n_t=k}\right]=\e^0\left[ \e\left[f(Y^{(k+1)}_{t-T_k})\ind_{\tau_{k+1}>t-T_k}\right] \ind_{T_k<t}\right].
\eeq
\end{corollary}
%Then according to Proposition 5.4.3.1 of \cite{MR2568861}, $p_0(t,x)$ satisfies PDE
%\begin{equation}
%\label{pde_p0}
%\left\{
%\begin{aligned}
%&\frac{\partial }{ \partial t}u = \frac{\partial }{\partial x}(xu)+\frac{\partial ^2}{\partial x^2}u
%\\
%&u(1,t)=0, \quad u(x,0)=\delta_0(x).
%\end{aligned}
%\right.
%\end{equation}
%All calculations in the rest of this section are done under the following assumptions on the PDE \eqref{pde_p0}:
%\begin{assumption}
%\label{assumption_regularity}
%The PDE \eqref{pde_p0} has a strong solution $p_0(t,x)$ such that such that for all $x_0\in (-\infty,0)\cup (0,1)$,
%$$
%\|p_0(\cdot,x_0)\|_\infty, \ \left\|\frac{\partial p_0(\cdot,x_0)}{\partial x}\right\|_\infty, \  \left\|\frac{\partial^2 p_0(\cdot,x_0)}{\partial x^2}\right\|_\infty<\infty.
%$$
%\end{assumption}
%
Note that
\beq
\label{integral}
F_{\tau_{1}}(t)=1- P(\tau_1>t)=1-F_0(1, t)=1-\int_{-\infty}^1 f_0(x, t)dx
\eeq
and
\beq
\label{jump time convolution}
F_{T_{n}}=F_{\tau_{1}}\ast F_{\tau_{2}}\ast \cdots \ast F_{\tau_{n}}.
\eeq
Moreover, for each $n$, $F_n(\cdot, t)$ is absolutely continuous and let $f_n(x,t)$ denotes its density.

In the rest of this section, we use Proposition \ref{proposition_strong} and the similar renewal argument as in \cite{MR63607} to calculate the distribution of $X_t$. First one has the following lemma
\begin{lemma}
\label{CDF iteration lemma}
For all $n\ge 1, \ t>0$, and $x<1$,
\bae
\label{CDF_iterate}
F_n(x,t)=\int_0^t F_{n-1}(x, t-s)dF_{\tau_1}(s).
\eae
Moreover, $F_n(x,t)$ is also bi-variate continuous on $(-\infty,1]\times(0,\infty)$.
\end{lemma}

\begin{proof}
Suppose the lemma holds for $n-1\ge 0$, which has been shown true for $n=1$. By Proposition \ref{proposition_strong}, Lemma \ref{CDF_continuity}, and Fubini's formula
%\textcolor{red}{(I do not delete Lemma \ref{CDF_continuity} because Fubini theorem relies on the measurability of function which is obtained by the continuity)}
on the independent random variables $T_{n-1}$ and $\tau_n$
$$
\begin{aligned}
F_n(x,t)&=\prob(X_t\le x, n_t=n)=\e_{0}\left[ \prob\left(Y^{(n+1)}_{t-T_n}\le x, \tau_{n+1}>t-T_n\right)\mathbbm{1}_{T_n<t} \right]\\
&=\e_{0}\left[ F_0(x, t-T_n)\mathbbm{1}_{T_n<t} \right]=\e_{0}\left[ F_0(x, t-T_{n-1}-\tau_n)\mathbbm{1}_{T_{n-1}+\tau_n<t} \right]\\
&=\int_0^t\int_0^{t-s} F_{0}(x, t-s-h) dF_{T_{n-1}}(h) dF_{\tau_1}(s)\\
&=\int_0^t F_{n-1}(x, t-s)dF_{\tau_1}(s)
\end{aligned}
$$
and thus we have got \eqref{CDF_iterate}. With \eqref{CDF_iterate}, for any $t_0>0$ and $x_0<1$, the continuity of $F_n(x,t)$ at $(x_0, t_0)$ with respect to $t$ can be shown as follows: For any $\ep>0$, by the continuity of $F_{\tau_1}(t)$, there is a $\delta_1\in (0,t_0)$ such that
$$
F_{\tau_1}(t_0+\delta_1)-F_{\tau_1}(t_0-\delta_1)<\ep.
$$
Now note that $F_{n-1}(x_0, t)$ is continuous on $(0,\infty)$ and thus uniformly continuous on $[\delta_1/2,t_0+\delta_1]$. Thus there is a $\delta_2>0$ such that for all $t_1,t_2\in [\delta_1/2,t_0+\delta_1]$, $|t_1-t_2|<\delta_2$,
$$
|F_{n-1}(x_0, t_1)-F_{n-1}(x_0, t_2)|<\ep.
$$
Thus for any $t$ such that $|t-t_0|<\min\{\delta_1/2,\delta_2\}$ (here we may without loss of generality assume that $t<t_0$), one has
\begin{align*}
	|F_{n}(x_0, t_0)-F_{n}(x_0, t)|&\le \int_0^{t_0-\delta_1} |F_{n-1}(x_0, t_0-s)-F_{n-1}(x_0, t-s)| dF_{\tau_1}(s)\\
	&+\int_{t_0-\delta_1}^t F_{n-1}(x_0, t-s) dF_{\tau_1}(s)+\int_{t_0-\delta_1}^{t_0} F_{n-1}(x_0, t_0-s) dF_{\tau_1}(s)\\
	&\le \ep+ 2[F_{\tau_1}(t_0+\delta)-F_{\tau_1}(t_0-\delta)]\le 3\ep.
\end{align*}

Similarly, the continuity of $F_{n}(x,t)$ at $(x_0, t_0)$ with respect to $x$ is guaranteed by that $F_{n-1}(x,t)$ is continuous and thus uniformly continuous on $[x,x']\times[\ep, t]$ for all $\ep>0$ and that $F_{\tau_1}(\cdot)$ put no mass on point $t_0$. And with the similar argument in the last lemma to show $F_n(\cdot, \cdot)$ is bi-variate continuous, we complete the proof.

\end{proof}

With the same argument as before, we  have

\begin{corollary}
\label{col_function_iteration_2}
For any bounded measurable function $f$, any integer $k\ge 1$ and any $t>0$,
$$
\e_0\left[f(X_t)\ind_{n_t=k}\right]=\int_{-\infty}^{1} f(x) dF_k(x, t)
$$
is measurable with respect to $t$, and
\beq
\e_0\left[ f(X_t) \ind_{n_t=k}\right]=\int_{-\infty}^{1} f(x) dF_k(x, t)=\int_0^t \int_{-\infty}^{1} f(x) dF_{k-1}(x, t-s) dF_{T_1}(s).
\eeq
\end{corollary}

Our next lemma gives the exponential decay of $F_{n}(x,t)$ on a compact set of $t$, which is useful in our later calculations especially when we need to deal with the convergence of some series.

\begin{lemma}
	\label{lemma_exp}
	There is a $\theta>0$ such that $T\in (0,\infty)$
	\bae
	F_n(x, t)\le \exp(-\theta n+T)
    \label{eq_exp_decay}
	\eae
	for all $n\in \BN$, $t\le T$ and $x\in (-\infty,1]$.
\label{exp decay}
\end{lemma}

\begin{proof}
	For any  $t\le T$ and $x\in (-\infty,1]$,
	$$
	F_n(x,t)= \prob(X_t\le x, n_t=n)\le \prob(n_t\ge n)= \prob (T_n\le t)\le \prob(T_n\le T).
	$$
	Thus it suffices to show that
	$$
	\prob(T_n\le T)\le \exp(-\theta n+T).
	$$
	Now recalling that $T_n=\sum_{i=1}^n \tau_i\in (0,\infty)$, define
	$$
	Y_n=\exp(-T_n)\in (0,1)
	$$
	where by the independence of $\{\tau_i, \ i\ge1\}$
	$$
	\e[Y_n]=\left(\e[\exp(-\tau_1)] \right)^n.
	$$
	Note that for a.s. $\omega, Y_t^{(1)}(\omega)$ is a continuous trajectory, which implies $\tau_1(\omega)>0\ a.s. $. Thus we have $\prob(\tau_1>0)=1$, which implies
	$$
	\e[\exp(-\tau_1)]=\exp(-\theta)<1
	$$
	for some $\theta>0$. Then the desired result follows from the Markov inequality for $Y_n$ and the fact that $\{T_n\le T\}=\{Y_n\ge \exp(-T)\}$.
\end{proof}

\begin{remark}
	The upper bound found in Lemma \ref{lemma_exp} is clearly not sharp, although it suffices the purpose in the later context.
\end{remark}

In light of the properties of joint process $(X_t,n_t)$ defined in \eqref{process definition} above, we have a new perspective to investigate the distribution of $X_t$.  Let $F(x, t)$ denote the cumulative distribution function of $X_t$. Based on the number of jumping times, it admits the following decomposition
\bae
\label{eq_CDF_sequence}
F(x, t)=\sum_{n=0}^\infty F_n(x, t).
\eae
There are two major types of results that we could obtain from the decomposition above.

On one hand, we immediately get the wellposedness and regularity properties of the distribution of $X_t$ at a given time, which are not easily achievable due to the complication of jumps.  We observe the right hand side of \eqref{eq_CDF_sequence} converges by the bounded convergence theorem, and, moreover, it is clear that by the previous lemmas $F(x, t)$ is continuous on $(-\infty,1]\times(0,\infty)$. Besides, due to the exponential decay of $F_n(x, t)$ with respect to $n$, we know that the measure induced by $F(\cdot, t)$ is absolutely continuous with respect to the Lebesgue measure, whose density function we shall denote by $f(x,t)$.

On the other hand, such a decomposition provides an auxiliary degree of freedom in the representation of the density function, which facilitates analyzing the time evolution of the density function. While the flux shift mechanism makes the evolution of $F(x,t)$ nonlocal, the decomposition unfold the distribution by adding one more dimension such that the evolution has a simpler structure: the evolution of $F_0$ is self-contained without any nonlocality, and for $n\ge 1$, the evolution of $F_n$ is also local, although it has a tractable dependence on $F_{n-1}$.  Recall that, we have used $f_n(x,t)$ to denote the density function of $F_n(x,t)$ respectively. In fact, we are able to show that $f_n(x,t)$ is a solution to a sub-PDE problem, and eventually, the exponential convergence in $n$ can help conclude that
\beq
\label{density decomposition}
f(x,t)=\sum_{n=0}^\infty f_n(x,t)
\eeq
is a solution of the PDE problem of interest satisfying the properties in Theorem \ref{thm_forward}.

%\textcolor{red}{Obviously, we have $f(x,t)=\sum_{n=0}^{\infty}f_n(x, t)$.}

\section{Iteration Approach}\label{iteration}
In this section we aim to prove the theorems in Section \ref{main result}. First, we prove the density of the process $X_t$ that starts from $0$ is an instantaneous smooth mild solution of \eqref{eq_planck} with initial condition $f_{\text{in}}(x)=\delta(x)$. Then with similar treatment we can get Corollary \ref{thm_forward y} easily, which together with the integral representation \eqref{convolution} derive Theorem \ref{thm_forward nu}. Finally, we show that the mild solution is consistent with the definition of the weak solution of \eqref{eq_planck} defined in \cite{caceres2011analysis}.

\subsection{Solutions in Iteration}\label{solutions}
\

\noindent

%Note that the density $f(x, t)=\sum_{n=0}^{+\infty} f_n(x, t)$.
Recalling the process $(X_t,n_t)$ defined in \eqref{process definition} above, we first focus on the case $X_0=0$, i.e. the initial condition PDE \eqref{eq_planck} is $f(x,0)=\delta(x)$. In the previous section, we have decomposed the distribution $F(x,t)$ of the stochastic process $X_t$ into a summation of series $\{F_n(x,t)\}_{n=0}^{+\infty}$ according to \eqref{iteration_cdf} and \eqref{eq_CDF_sequence}. We also decompose the original PDE problem \eqref{eq_planck} into a sequence of sub-PDE problems: for $n=0$
\begin{equation}
\label{classical f0}
\left\{
\begin{aligned}
&\frac{\partial f_0}{ \partial t} - \frac{\partial }{\partial x}\left ( x f_0\right) - \frac{\partial ^2 f_0}{\partial x^2} =0, \quad x \in (-\infty, 1),t\in(0,T],\\
&f_0(-\infty,t)=0, \quad f_0(1,t)=0, \quad t \in [0, T],\\
&f_0(x,0)=\delta(x) \quad  \mbox{in}\ \CD'(-\infty, 1)
\end{aligned}
\right.
\end{equation}
where $\CD(-\infty, 1)=C_c^\infty(-\infty,1)$ and for $n \ge 1$ define $N_{n-1}(t)=-\frac{\partial}{\partial x}f_{n-1}(1,t)$, we solve
\begin{equation}
\label{classical fn}
\left\{
\begin{aligned}
&\frac{\partial f_n}{ \partial t} - \frac{\partial }{\partial x}\left ( x f_n\right) - \frac{\partial ^2 f_n}{\partial x^2} =0, \quad x \in (-\infty,0) \cup (0,1), t \in(0,T],\\
&f_n(0^-,t)=f_n(0^+,t), \quad  \frac{\partial}{\partial x}f_n(0^-,t)-\frac{\partial}{\partial x}f_n(0^+,t)=N_{n-1}(t),\quad t\in (0,T],\\
&f_n(-\infty,t)=0, \quad f_n(1,t)=0, \quad t\in [0,T],\\\
&f_n(x,0)=0, \quad x\in(-\infty,1).
\end{aligned}
\right.
\end{equation}
In particular, we find the PDE problem for $f_0$ \eqref{classical f0} is self-contained with a singular initial data, and thus only a mild solution can be expected, which, however, can be shown to be instantaneously smooth. For $n\ge 1$ the PDE problems for $f_n$ \eqref{classical fn} are defined when $x \in  (-\infty,0) \cup (0,1)$, and the time-dependent interface boundary data $N_{n-1}$ at $x=1$ is determined by $f_{n-1}$, the solution to the previous PDE problem in the sequence, but the classical solution of such problems can be understood in the usual sense.

Here is a bit ambiguity in the notations, since we have used $f_n(x,t)$ to denote the sub density function of the stochastic process and also the solution to the PDE problem. In fact, we shall show those two functions coincide, of which the precise meaning shall be specified.   In the following,  we show that sub density function $f_0$ with delta initial data is an instantaneous smooth mild solution of \eqref{classical f0}, and then following the iteration scheme, we prove that for each $n \ge 1$, the sub density function $f_n$ is the classical solution of \eqref{classical fn}. We conclude with the proof of Theorem \ref{thm_forward} by the end of this section.
%In our calculations, the regularity of $f_0$ is essential. We have the following

Before we start to prove our main theorem, we first discuss the Green function of the Fokker-Planck equation \eqref{classical f0}. According to Theorem $1.10$ in Chapter $\uppercase\expandafter{\romannumeral 6}$ of \cite{garroni1992green} by Garroni and Menaldi, we know that the generator of the O-U process \eqref{OU_sol}, i.e.,
$$
\mathcal L_y:=(-y)\partial_y\cdot+\partial^2_{yy}\cdot,
$$
admits a Green's function $G: (-\infty, 1]\times[0,T]\times(-\infty,1]\times[0,T] \ni (y,s,x,t)\mapsto G(y,s,x,t)$. For a given $(x,t)\in (-\infty,1]\times[0,T]$, the function $(-\infty, 1]\times[0,t)\ni (y,s)\mapsto G(y,s,x,t)$ is a solution of the PDE
\begin{equation}
\left\{
\begin{aligned}
&\partial_sG(y,s,x,t)+\mathcal L_y G(y,s,x,t)=0, \quad y \in (-\infty, 1),s \in[0,t),\\
&G(1,s,x,t)=0, \quad s \in [0, t],\\
&G(y,t,x,t)=\delta(y-x)\  \mbox{in}\ \CD'(-\infty, 1)
\end{aligned}
\right.
\end{equation}
Following Theorem 5 in Chap.9 of  \cite{friedman2008partial}, for a given $(y,s)\in (-\infty,1)\times[0,T)$, the function $(-\infty, 1]\times(s,T]\ni (x,t)\mapsto G(y,s,x,t)$ is also known to be Green's function of the adjoint operator
$$
\mathcal L_x^*=\partial_x[x\cdot]+\partial^2_{xx}\cdot,
$$
i.e. the function $(-\infty, 1]\times(s,T]\ni (x,t)\mapsto G(y,s,x,t)$ is a classical solution of the PDE
\begin{equation}
\label{forward equation}
\left\{
\begin{aligned}
&\partial_tG(y,s,x,t)=\mathcal L_x^*G(y,s,x,t), \quad x \in (-\infty, 1),t \in(s,T],\\
&G(y,s,1,t)=0, \quad t \in [s, T],\\
&G(y,s,x,s)=\delta(x-y)\  \mbox{in}\ \CD'(-\infty, 1),
\end{aligned}
\right.
\end{equation}
which is consistent with \eqref{classical f0}. Now we give an important lemma that connects the density function of the stochastic process before the first jumping time with the Green function of PDE problem \eqref{classical f0}, which is the starting point of our iteration strategy. And for Green function $G$, although we can not find a closed formula for it, there exists the following estimation.

\begin{lemma}
\label{Green function estimation}
There exists a unique Green function $G: (-\infty, 1]\times[0,T]\times(-\infty,1]\times[0,T] \ni (y,s,x,t)\mapsto G(y,s,x,t)$ for equation \eqref{classical f0}. Let $f_0(x,t)$ denotes the density of the distribution $F_0(x,t)$ defined in \eqref{iteration_cdf}, then $f_0(x,t)=G(0,0,x,t)$, i.e., it is a mild solution of \eqref{classical f0} on $(-\infty, 1]\times[0,T]$. Besides, we have the estimation:
\begin{equation}
    \label{Gest}
    \left|\partial^\ell G(y,s,x,t)\right|\le C(t-s)^{-\frac{1+\ell}{2}}\exp\left(-C_0\frac{(x-y)^2}{t-s}\right),\quad 0\le s<t\le T.
\end{equation}
where $\ell=0,1,2$, $\partial^{\ell}=\partial_{t x}^{\ell}=\partial_{t}^{m} \partial_{x}^{n}, \,\ell=2 m+n,$ for $m,n \in \mathbb N_0$.
\end{lemma}
\begin{proof}
Set
$$
p(x,t):=G(0,0,x,t),\quad \ x\in(-\infty,1], t\in (0,T].
$$
now we prove that $p(x,t)$ coincides with $f_0(x,t)$, which immediately derives that $f_0(x,t)$ is a mild solution of equation \eqref{classical f0}. Given a smooth function $\phi: (-\infty,1]\times[0,T]\to \BR$ with a compact support, noting that Green's function satisfies \eqref{forward equation}, we have that the PDE
\begin{equation}
\label{special equation}
\left\{
\begin{aligned}
\partial_s u(y,s)-y\partial_y u(y,s)+\partial_{yy} u(y,s)+\phi(y,s)&=0,\quad (y,s)\in (-\infty,1)\times(0,T]\\
u(1,s)&=0, \quad s\in[0,T],\\
u(y,T)&=0 \quad y\in(-\infty,1)
\end{aligned}
\right.
\end{equation}
admits a (unique) classical solution
\beq \label{r1}
u(y,s)=\int_s^T\int_{-\infty}^1G(y,s,x,t)\phi(x,t)dxdt,\quad s\in[0,T),\ y\le1.
\eeq
Moreover, u is bounded and continuous on $(-\infty,1]\times[0,T]$ and is once continuously differentiable in time and twice differentiable in space on $(-\infty,1)\times[0,T]$. Let $(X_t,n_t)$ be the process defined in \eqref{process definition} and $\tau:=\inf\{t\ge0: X_{t\wedge T}\ge 1\}$. By $It\hat{o}'s$ formula, we have
$$
du(X_{t\wedge \tau},t\wedge \tau)=-\phi(X_{t \wedge \tau},t \wedge \tau)dt+\sqrt{2}u_x(X_{t \wedge \tau},t \wedge \tau)dB_t
$$
Integrating above formula from $0$ to $T$ and take the expectation, with the boundary condition in \eqref{special equation}, we then have the representation formula:
\beq \label{r2}
u(0,0)=\e\left[\int_0^{T\wedge\tau}\phi(X_t,t)dt\right]
\eeq
And with the two presentations for $u(0,0)$ above, i.e. \eqref{r1} and \eqref{r2}, we obtain
\[
\e\left[\int_0^{T\wedge\tau}\phi(X_t,t)dt\right]=\int_0^T\int_{-\infty}^1p(x,t)\phi(x,t)dxdt.
\]
We further rewrite \eqref{r2} as follows.
$$
\e\left[\int_0^{T\wedge\tau}\phi(X_t, t)dt\right]=\int_0^T \e\left[\phi(X_t, t)\mathbbm{1}_{\{t\le \tau\}}\right]dt=\int_0^T\int_{-\infty}^1\phi(x,t)\prob(X_t\in dx,\tau>t)dt,
$$
Clearly, for $t\in [0,T]$, $\{\tau> t\}=\{T_1> t\}=\{n_t=0\}$ and thus
\beq
\label{consistent}
\int_0^T\int_{-\infty}^1\phi(x,t)f_0(x,t)dxdt=\int_0^T\int_{-\infty}^1\phi(x,t)p(x,t)dxdt
\eeq
By \eqref{decay at infinity} and \eqref{Gest}, $p(x,t)$ and $f_0(x,t)$ decay at $-\infty$ and thus \eqref{consistent} is also valid for any smooth function $\phi$ that is only bounded, which derives that the density function $f_0(x,t)$ coincides with $p(x,t)$. With \eqref{eq_initial_delta}, we conclude that $f_0(x,0)=\delta(x)$ and thus $f_0(x,t)$ is a mild solution of \eqref{classical f0}. The complete proof of estimation \eqref{Gest} can be found in Theorem $1.10$ in Chapter $\uppercase\expandafter{\romannumeral 6}$ of \cite{garroni1992green} by Garroni and Menaldi and the proof is complete.

\end{proof}
\begin{remark}
The proof of Lemma \ref{Green function estimation} is essentially implied from the results in \cite{delarue2013first,delarue2015global,garroni1992green}, in particular, Lemma 2.1 of \cite{delarue2013first} and Theorem $1.10$ in Chapter $\uppercase\expandafter{\romannumeral 6}$ of \cite{garroni1992green}.
\end{remark}
Next, we prove some regularities of the sub-density $f_0(x,t)$ that are useful in our later calculations.
\begin{proposition}
\label{regularity f0}
Let $X_t$ be the process defined in \eqref{process definition} and $T_1$ be the stopping time defined in \eqref{stopping time}. Let $F_0(x,t)$ be defined in \eqref{iteration_cdf} and its density is denoted as $f_0(x,t)$. Let $f_{T_1}(t)$ denotes the p.d.f. of $T_1$. For any fixed $T>0$, we have
\begin{itemize}
    \item[(\romannumeral1)]
    \beq
    \label{decay 0}
    \lim_{x \to -\infty }\partial_xf_0(x, t)=0, \quad t\in (0,T].
    \eeq
    \item[(\romannumeral2)]
    For any $x_{0}\in(0,1)$, $f_{0}(x,t)\in C^{2, 1} \left((-\infty, -x_0]\cup [x_0,1]\times [0,T]\right)$.
    Moreover for all $|x|\ge |x_{0}|$, $\lim_{t\to 0^{+}}f_{0}(x,t)=0$.
    \item[(\romannumeral3)]
    For any $0<\varepsilon_{0}<T<\infty$, $f_0(x,t)\in C^{2, 1}\left((-\infty,1]\times [\varepsilon_{0},T]\right)$. With the following uniform gradient estimations
    \begin{equation}
    \begin{aligned}
    &\sup_{(-\infty,1]\times[\varepsilon_{0},T]}\left|f_{0}\right|< \infty, \quad \sup_{(-\infty,1]\times[\varepsilon_{0},T]}\left|\frac{\partial f_0}{\partial t}\right|< \infty, \quad \sup_{(-\infty,1]\times[\varepsilon_{0},T]}\left|\frac{\partial f_{0}}{\partial x}\right| < \infty,\\
    &\sup_{(-\infty,1]\times[\varepsilon_{0},T]}\left|\frac{\partial (x f_{0})}{\partial x}\right| < \infty, \quad \sup_{(-\infty,1]\times[\varepsilon_0,T]}\left|\frac{\partial^2 f_{0}}{\partial x^2}\right|<\infty.
    \end{aligned}\label{regularity 3}
    \end{equation}
    \item[(\romannumeral4)]
    We have the coupling relation between $f_{T_{1}}(t)$ and $f_0(x,t)$: $\forall t \in (0,T]$, it satisfies
    \begin{equation}
    \label{T_1 pdf}
    f_{T_{1}}(t)=-\int_{-\infty}^{1}\frac{\partial f_{0}(x,t)}{\partial t}dx=-\frac{\partial}{\partial x} f_0(1,t)
    \end{equation}
    and $f_{T_1}(t)\in C[0,T]$ with $f_{T_1}(0)=0$.
\end{itemize}
\end{proposition}
\begin{proof}
 (\romannumeral1) is the direct corollary of estimate \eqref{Gest}. And from \eqref{Gest}, we know that the Green function of \eqref{classical f0} is continuous differentiable and decays exponentially fast as $t$ tends to $0^+$ when $x$ stay away from $0$. Thus we immediately obtain the properties in (\romannumeral2). Also by the estimation \eqref{Gest} for the Green function, we can easily get the bound for $f_0$ in (\romannumeral3) when $t$ stay away from $0$. Finally, to prove (\romannumeral4), recall that $f_0(x,t)dx=\BP(X_t\in dx,T_1>t)$, thus the c.d.f of $T_1$ is given by
$$
\prob(T_1\le t)=1-\prob(T_1>t)=1-\int_{-\infty}^1f_0(x,t)dx.
$$
By \eqref{Gest}, we can differentiate the above formula w.r.t $t$ and exchange the derivative and the integral. Using (\romannumeral1) and the boundary condition of $f_0$, we have for any $t \in (0,T]$,
\begin{equation*}
f_{T_{1}}(t)=\frac{d}{dt}\prob(T_1\le t)=-\int_{-\infty}^{1}\frac{\partial f_{0}(x,t)}{\partial t}dx=-\int_{-\infty}^{1}\frac{\partial }{\partial x}\left ( x f_0\right) + \frac{\partial ^2 f_0}{\partial x^2}dx=-\frac{\partial}{\partial x} f_0(1,t).
\end{equation*}
And with Lemma \ref{Green function estimation}:
$$
\left|f_{T_{1}}(t)\right|=\left|\partial_x f_0(1,t)\right|\le \frac{C}{t}\exp(-\frac{C_0}{t}),
$$
we conclude $f_{T_{1}}(t)\in C(0,T]$ and $\lim_{t\to 0^+}f_{T_1}(t)=0$ and thus $f_{T_{1}}(t)\in C[0,T]$.

\end{proof}

In order to make the iteration strategy successful,  we need to further show that $f_{T_1}(t)$ is continuously differentiable, which is not a direct consequence of  estimating Green's function. Thus next we shall prove that $f_{T_1}(t)\in C^1[0,T]$ and the following estimation is useful in the further calculations.
\begin{corollary}
\label{regularity (0,T]}
For any $T>0$ and $\forall 0<\varepsilon_0<\min\{\frac1T, T\}$, $f_{T_1}(t)\in C^1(0,T]$ and for any $t\ge \varepsilon_0$, we have
\beq
\label{epsilon estimation}
\left|f'_{T_1}(t)\right|\le C\varepsilon_0^{-3}.
\eeq
\end{corollary}
\begin{proof}
By Proposition \ref{regularity f0}, we know that $f_0(x,t)\in C^{2, 1}\left((-\infty,1]\times [\varepsilon_{0},T]\right)$ and $f_{T_1}(t)=-\frac{\partial}{\partial x} f_0(1,t)\in C[0,T]$. Then for any $x\in(-\infty,1]$, $t \in [\varepsilon_0, T]$, set $g_0(x,t)=\frac{\partial}{\partial t}f_0(x,t)$ and it satisfies
\begin{equation}
\label{classical g0}
\left\{
\begin{aligned}
&\frac{\partial g_0}{ \partial t} - \frac{\partial }{\partial x}\left ( x g_0\right) - \frac{\partial ^2 g_0}{\partial x^2} =0, \quad x \in (-\infty, 1), t \in (\varepsilon_0, T],\\
&g_0(-\infty,t)=0, \quad g_0(1,t)=0, \quad  t \in [\varepsilon_0, T],\\
&g_0(x,\varepsilon_0)=\frac{\partial}{\partial t}f_0(x,\varepsilon_0) \quad x \in (-\infty, 1).
\end{aligned}
\right.
\end{equation}
Defining $\varphi(x):=\frac{\partial}{\partial t}f_0(x,\varepsilon_0)$, we immediately get that $\varphi(x) \in C^2(-\infty,1]\cap L^\infty(-\infty,1]$ and by \eqref{Gest}
$$
|\varphi(x)|\le C\varepsilon_0^{-\frac32}.
$$
For any $t\ge 0$, $x\in (-\infty,1]$, define $h(x,t):=g_0(x,t+\varepsilon_0)$ and then $h(x,0)=\varphi(x)$. Recalling the Green function $G(s,y,x,t)$ in PDE \eqref{forward equation}, we have
\begin{equation*}
h(x,t)=\int_{-\infty}^{1}G(y,0,t,x)\varphi(y)dy,\quad t\ge 0.
\end{equation*}
Then
\beq
g_0(x,t)=\int_{-\infty}^{1}G(y,0,t-\varepsilon_0,x)\varphi(y)dy,\quad t\ge \varepsilon_0.
\eeq
By \eqref{Gest} and Lemma \ref{regularity f0}, we have
\beq
f'_{T_1}(t)=-\frac{\partial }{\partial t}\frac{\partial}{\partial x}f_0(1,t)=-\frac{\partial}{\partial x}g_0(1,t)=-\int_{-\infty}^{1}\frac{\partial}{\partial x}G(y,0,t-\varepsilon_0,1)\varphi(y)dy,\quad t>\varepsilon_0
\eeq
and thus $f_{T_1}(t)\in C^1(\varepsilon_0,T]$.\\
When $t\ge 2\varepsilon_0$,
\begin{equation}
\begin{aligned}
|f'_{T_1}(t)|&\le \int_{-\infty}^{1}\left|\frac{\partial}{\partial x}G(y,0,t-\varepsilon_0,1)\right| C\varepsilon_0^{-\frac32}dy\\
&\le C\varepsilon_0^{-\frac32} \int_{-\infty}^{1} \frac{C}{t-\varepsilon_0}\exp\left(-C_0\frac{(1-y)^2}{t-\varepsilon_0}\right) dy\\
&= C\varepsilon_0^{-\frac52}\int_0^{+\infty}\exp\left( -C_0 \frac{\xi^2}{t-\varepsilon_0}\right)d\xi\\
&\le C\varepsilon_0^{-\frac52} \frac{\sqrt{T-\varepsilon_0}}{\sqrt{C_0}}\frac{\sqrt\pi}{2}\\
&\le C\varepsilon_0^{-3}.
\end{aligned}
\end{equation}
where the second inequality is by the change of variable $\xi=1-y$ and the third inequality is from the fact $\varepsilon_0\le \frac1T$. And because $\varepsilon_0$ can be arbitrarily small, we complete the proof.

\end{proof}

Now we focus on the behavior of $f'_{T_1}(t)$ when $t$ is small. This proof is partially inspired by the reformulation and  the representation proposed in \cite{carrillo2013classical}.

\begin{proposition}
\label{regularity [0,T]}
The p.d.f. $f_{T_1}(t)$ of the first hitting time $T_1$ is $C^1[0,T]$ for any fixed $T>0$.
\end{proposition}
\begin{proof}
By Proposition \ref{regularity f0} and Corollary \ref{regularity (0,T]}, we know $f_{T_1}(t)\in C^1(0,T]\cap C[0,T]$ and thus we only need to prove that $\lim_{t\to 0^+}f'_{T_1}(t)$ exists. We prove it in the following steps. \begin{itshape}Step 1:\end{itshape} We rewrite the problem \eqref{classical f0} as a moving boundary problem and rewrite $f_{T_1}(t)$ as $M(s)$. With the heat kernel $\Gamma$, we derive an integral representation of $M(s)$. \begin{itshape}Step 2:\end{itshape} We analyse the decay rate of $M(s)$ and $M'(s)$ at $0$ by utilizing the decay property of heat kernel $\Gamma$.  \begin{itshape}Step 3:\end{itshape} Using the estimations of $M(s)$, $M'(s)$ and heat kernel $\Gamma$, we derive $\lim_{t\to 0^+}f'_{T_1}(t)=0$.

\begin{itshape}Step 1:\end{itshape} Inspired from \cite{carrillo2013classical}, we introduce a change of variable to transform \eqref{classical f0}  to a moving boundary problem. Let
\bae
\label{change of variable}
y=e^{t}x,\quad s=(e^{2t}-1)/2, \quad u(y,s) = e^{-t} f(x,t).
\eae
Note that PDE \eqref{classical f0} is for the O-U process killed at a stopping time and thus has the Dirichlet boundary condition. By the standard change of variable \eqref{change of variable}, we can transform \eqref{classical f0} into a heat equation with the moving boundary $b(s)=\sqrt{2s+1}$. Actually, we have the new equation
\begin{equation}
\label{classical f0 change}
\left\{
\begin{aligned}
&u_s=u_{yy}, \quad y\in (-\infty, b(s)), s>0\\
&u(-\infty,s)=0, \quad u(b(s), s)=0, \quad s\ge 0,\\
&u(y,0)=\delta(y)\ \mbox{in}\ \CD'(-\infty, b(s)).
\end{aligned}
\right.
\end{equation}
Let $\Gamma$ be the Green's function for the heat equation on the real line:
\beq
\label{Green function}
\Gamma(y,s,\xi,\tau)=\frac{1}{\sqrt{4\pi(s-\tau)}}\exp\{-\frac{(y-\xi)^2}{4(s-\tau)}\},
\quad s>\tau.
\eeq
In the region $-\infty <\xi <b(\tau),\ 0<\tau<h $, recall the  Green's identity
\beq
\label{green identity}
\frac{\partial}{\partial \xi}(\Gamma u_\xi-u\Gamma_\xi)-\frac{\partial}{\partial \tau}(\Gamma u)=0.
\eeq
To derive an expression of $u$, we consider the integration of \eqref{green identity} over such a region and let
\begin{equation*}
\begin{aligned}
&\uppercase\expandafter{\romannumeral 1}=\int_0^s\int_{-\infty}^{b(\tau)}(\Gamma u_\xi)_\xi d\xi d\tau, \quad
\uppercase\expandafter{\romannumeral 2}=\int_0^s\int_{-\infty}^{b(\tau)}(u\Gamma_\xi)_\xi d\xi d\tau, \quad
\uppercase\expandafter{\romannumeral 3}=\int_0^s\int_{-\infty}^{b(\tau)}(\Gamma u)_\tau d\xi d\tau.
\end{aligned}
\end{equation*}
We have
$$
\uppercase\expandafter{\romannumeral 1}=\int_0^s \Gamma u_\xi|_{\xi=b(\tau)}d\tau.
$$
Using the boundary condition of $u(y, s)$ in \eqref{classical f0 change}, we have
$$
\uppercase\expandafter{\romannumeral 2}=0
$$
and
$$
\uppercase\expandafter{\romannumeral 3}=\int_{-\infty}^{b(s)}\Gamma u|_{\tau=s^-} d\xi- \int_{-\infty}^{b(0)}\Gamma u|_{\tau=0} d\xi=u(y,s)-\int_{-\infty}^{b(0)}\Gamma u|_{\tau=0} d\xi.
$$
Plugging in \eqref{green identity},

\begin{equation}
\begin{aligned}
u(y,s)&=\int_{-\infty}^{b(0)}\Gamma(y,s,\xi,0) \delta(\xi) d\xi
+\int_0^s \Gamma(y,s,b(\tau),\tau)u_\xi(b(\tau), \tau)d\tau\\
&=\Gamma(y,s,0,0)-\int_0^s \Gamma(y,s,b(\tau),\tau)M(\tau)d\tau.
\label{expression}
\end{aligned}
\end{equation}
where $M(\tau)=-u_\xi(b(\tau), \tau)$.
Note that the Green function $\Gamma$ is infinitely continuously differentiable, thus the regularity of $u$ depends on $M$. Using Lemma $1$ on Page $217$ of \cite{friedman2008partial}, we know that for any continuous function $\rho$ the following limit holds:
$$
\lim_{y \to b(s)^-} \frac{\partial}{\partial y}\int_0^s \rho(\tau)\Gamma(y,s,b(\tau),\tau)d\tau=\frac 12\rho(s)+\int_0^s\rho(\tau) \Gamma_y(b(s),s,b(\tau),\tau)d\tau.
$$
So differentiating \eqref{expression} at $y=b(s)^-$, we can get the following integral equation
$$
-M(s)=\Gamma_y(b(s),s,0,0)-\frac 12 M(s)-\int_0^s \Gamma_y(b(s),s,b(\tau),\tau)M(\tau)d\tau.
$$
That is
\beq
\begin{aligned}
\label{M equation}
M(s)&=-2\Gamma_y(b(s),s,0,0)+2\int_0^s \Gamma_y(b(s),s,b(\tau),\tau)M(\tau)d\tau\\
&=:2J_1(s)+2J_2(s).
\end{aligned}
\eeq
Recalling the change of variable in \eqref{change of variable} and taking derivatives directly, we know that
\beq
\label{relationship of density}
f_{T_1}(t)=e^{2t}M(s) \quad \text{and} \quad f'_{T_1}(t)=2e^{2t}M(s)+e^{4t}M'(s).
\eeq

\begin{itshape}Step 2:\end{itshape} We shall analyse the decay rate of $M(s)$ at $0$. By heat kernel \eqref{Green function} $\Gamma(y,s,0,0)=\frac{1}{\sqrt{4\pi s}}\exp(\frac{-y^2}{4s})$ and $b(s)=\sqrt{2s+1}$, we have that for any $n\ge 0$, $\lim_{s \to 0^+}\frac{J_1(s)}{s^n}=0$ and thus there exists a constant $C$ such that for $s \in[0,T]$, $n\ge 0$
\beq
\left|J_1(s)\right|\le C s^n.
\label{J1 estimate}
\eeq
Note that
\beq
\label{formula}
\Gamma_y(b(s),s,b(\tau),\tau)=\frac{1}{\sqrt{4\pi(s-\tau)}}\exp\{-\frac{(b(s)-b(\tau))^2}{4(s-\tau)}\}
\{\frac{b(s)-b(\tau)}{-2(s-\tau)}\},
\eeq
thus we have
$$
|\Gamma_y(b(s),s,b(\tau),\tau)| \le \frac{C}{(s- \tau)^{\frac 12}}.
$$
By (\romannumeral4) of Proposition \ref{regularity f0} and \eqref{relationship of density}, there exists another big enough constant $K$ s.t. $\left|M(s)\right|\le K$, $\forall s\in[0,T]$ . Thus
$$
\left|J_2(s)\right|\le C\int_0^s\frac{K}{(s- \tau)^{\frac 12}}=C\sqrt{s}.
$$
Combining with \eqref{J1 estimate}, we also have $\left|M(s)\right|\le \left|J_1(s)\right|+\left|J_2(s)\right|\le C\sqrt{s}$, and thus
$$
\left|J_2(s)\right|\le C\int_0^s\frac{\sqrt{\tau}}{(s- \tau)^{\frac 12}}=Cs.
$$
Using \eqref{J1 estimate} again, we have $\left|M(s)\right|\le Cs$ and thus
$$
\left|J_2(s)\right|\le C\int_0^s\frac{\tau}{(s- \tau)^{\frac 12}}=Cs^{\frac32}.
$$
Using \eqref{J1 estimate} for the third time, we can get $\left|M(s)\right|\le Cs^{\frac32}$, which together with $M(0)=0$ leads to the right derivative of M at $0$ exists and
$$
M'(0^+)=\lim_{s\to 0^+}\frac{M(s)}{s}=0.
$$
Repeating the above calculations step by step, we can get for any $n\ge 0$, there exists a constant that depends on $n$, such that
\beq
\label{M order n}
\left|M(s)\right|\le Cs^n.
\eeq
By \eqref{epsilon estimation} and \eqref{relationship of density}, we know that for any sufficiently small $\varepsilon_0>0$ , there is a constant $C<+\infty$ such that
\beq
\label{epsilon estimation M}
\left|M'(s)\right|\le C\varepsilon_0^{-3},\quad \forall s\in[\varepsilon_0,1].
\eeq

\begin{itshape}Step 3:\end{itshape} In order to prove $f_{T_1}(t)\in C^1[0,T]$, which is equivalent to prove that $\lim_{s\to 0^+}M'(s)$ exists by \eqref{relationship of density}, now we prove that $\lim_{s\to 0^+}M'(s)=0$. Using \eqref{M equation} and the fact $\lim_{s \to 0^+}J'_1(s)=0$, we only need to to prove that
\beq
\label{goal}
\lim_{s \to 0^+}J'_2(s)=0.
\eeq
Using the estimations \eqref{M order n}, \eqref{epsilon estimation M} and heat kernel $\Gamma$, we compute the difference between $A:=\int_0^s \Gamma_y(b(s),s,b(\tau),\tau)M(\tau)d\tau$ and $B:=\int_0^{s+\Delta s} \Gamma_y(b(s+\Delta s),s+\Delta s,b(\tau),\tau)M(\tau)d\tau$.

$A$ can have the following decomposition
$$
A:=\left(\int_0^{\frac s2}+\int_{\frac s2}^s\right) \Gamma_y(b(s),s,b(\tau),\tau)M(\tau)d\tau.
$$
and for $B$,
\begin{equation*}
B:=\left(\int_0^{\frac s2}+\int_{\frac s2}^{\frac s2+\Delta s}+ \int_{\frac s2+\Delta s}^{s+\Delta s} \right) \Gamma_y(b(s+\Delta s),s+\Delta s,b(\tau),\tau)M(\tau)d\tau.
\end{equation*}
Define
$$
\frac{J_2(s+\Delta s)-J_2(s)}{\Delta s}=:I_1+I_2+I_3
$$
where
$$
I_1:=\int_0^{\frac s2} \left[\frac{\Gamma_y(b(s+\Delta s),s+\Delta s,b(\tau),\tau)-
\Gamma_y(b(s),s,b(\tau),\tau)}{\Delta s}\right]M(\tau)d\tau,
$$
$$
I_2:=\frac{1}{\Delta s}\int_{\frac s2}^{\frac s2+\Delta s}\Gamma_y(b(s+\Delta s),s+\Delta s,b(\tau),\tau)M(\tau)d\tau
$$
and
$$
I_3:=\frac{1}{\Delta s}\left[\int_{\frac s2+\Delta s}^{s+\Delta s} \Gamma_y(b(s+\Delta s),s+\Delta s,b(\tau),\tau)M(\tau)d\tau-\int_{\frac s2}^s \Gamma_y(b(s),s,b(\tau),\tau)M(\tau)d\tau\right].
$$
Thus to get \eqref{goal}, now it suffices to show that
\beq
\label{first term}
\lim_{\Delta s\to 0}\left|I_1\right|\le\int_0^{\frac s2} \left|\partial_s
\Gamma_y(b(s),s,b(\tau),\tau)M(\tau)\right|d\tau=o(1),
\eeq
\beq
\label{second term}
\lim_{\Delta s\to 0}\left|I_2\right|=o(1)
\eeq
and
\beq
\label{third term}
\lim_{\Delta s\to 0}\left|I_3\right|=o(1).
\eeq
The above $=o(1)$ means that the left side goes to $0$ as $s \to 0^+$.

Note that for $\tau \le \frac34 s$, then $\Gamma_y$ and $\partial_s\Gamma_y$ terms in \eqref{first term} and \eqref{second term} can be bounded by a polynomial order with respect to $s^{-1}$, which together with \eqref{M order n} immediately derives \eqref{first term} and \eqref{second term}. Thus we only need to focus on proving \eqref{third term}. With a simple change of variable, we have
\begin{equation*}
\begin{aligned}
&\int_{\frac s2+\Delta s}^{s+\Delta s} \Gamma_y(b(s+\Delta s),s+\Delta s,b(\tau),\tau)M(\tau)d\tau\\
=&\int_{\frac s2}^{s} \Gamma_y(b(s+\Delta s),s+\Delta s,b(\tau+\Delta s),\tau+\Delta s)M(\tau+\Delta s)d\tau\\
=&\int_{\frac s2}^{s} \Gamma_y(b(s+\Delta s),s+\Delta s,b(\tau+\Delta s),\tau+\Delta s)M(\tau)d\tau\\
+&\int_{\frac s2}^{s} \Gamma_y(b(s+\Delta s),s+\Delta s,b(\tau+\Delta s),\tau+\Delta s)\left[M(\tau+\Delta s)-M(\tau)\right]d\tau.
\end{aligned}
\end{equation*}
We define
$$
I_3:=I_{3,1}+I_{3,2}
$$
where
$$
I_{3,1}=\frac{1}{\Delta s}\int_{\frac s2}^{s} \left[\Gamma_y(b(s+\Delta s),s+\Delta s,b(\tau+\Delta s),\tau+\Delta s)- \Gamma_y(b(s),s,b(\tau),\tau)\right]M(\tau)d\tau
$$
and
$$
I_{3,2}=\int_{\frac s2}^{s} \Gamma_y(b(s+\Delta s),s+\Delta s,b(\tau+\Delta s),\tau+\Delta s)\frac{M(\tau+\Delta s)-M(\tau)}{\Delta s}d\tau.
$$
Thus to show \eqref{third term}, it suffices to prove
\beq
\label{fourth term}
\lim_{\Delta s\to 0}\left|I_{3,1}\right|=o(1)
\eeq
and
\beq
\label{fifth term}
\lim_{\Delta s\to 0}\left|I_{3,2}\right|=o(1).
\eeq
For \eqref{fourth term}, by \eqref{formula} we have
\begin{equation*}
\begin{aligned}
&\Gamma_y(b(s+\Delta s),s+\Delta s,b(\tau+\Delta s),\tau+\Delta s)-\Gamma_y(b(s),s,b(\tau),\tau)\\
&=\frac{1}{\sqrt{4\pi(s-\tau)}}\exp\{-\frac12\frac{b(s+\Delta s)-b(\tau+\Delta s)}{b(s+\Delta s)+b(\tau+\Delta s)}\}
\{\frac{-1}{b(s+\Delta s)+b(\tau+\Delta s)}\}\\
&-\frac{1}{\sqrt{4\pi(s-\tau)}}\exp\{-\frac12\frac{b(s)-b(\tau)}{b(s)+b(\tau)}\}
\{\frac{-1}{b(s)+b(\tau)}\},
\end{aligned}
\end{equation*}
and thus there exists a constant $C<+\infty$ independent of the choices of $s$, $\tau$ and $\Delta s$ such that
$$
\left|\Gamma_y(b(s+\Delta s),s+\Delta s,b(\tau+\Delta s),\tau+\Delta s)- \Gamma_y(b(s),s,b(\tau),\tau)\right|\le C\cdot \Delta s\cdot \frac{1}{\sqrt{s-\tau}},
$$
which together with \eqref{M order n} derive $\lim_{\Delta s\to 0}\left|I_{3,1}\right|=o(1)$.

Finally, for \eqref{fifth term}, note that $M(\tau)\in C^1[\frac s2, s]$ and that $|\Gamma_y(b(s+\Delta s),s+\Delta s,b(\tau+\Delta s),\tau+\Delta s)|\le \frac{C}{\sqrt{s-\tau}}$. By the dominated convergence theorem, we have the limit in \eqref{fifth term} exists and equals to
\beq
\label{sixth term}
I_{3,3}:=\int_{\frac s2}^{s} \Gamma_y(b(s),s,b(\tau),\tau)M'(\tau)d\tau.
\eeq
To prove $\left|I_{3,3}\right|=o(1)$, one may further decompose it as
\begin{equation*}
\begin{aligned}
I_{3,3}=&\int_{\frac s2}^{s-s^7} \Gamma_y(b(s),s,b(\tau),\tau)M'(\tau)d\tau+\int_{s-s^7}^{s} \Gamma_y(b(s),s,b(\tau),\tau)M'(\tau)d\tau\\
=&:I_4+I_5.
\end{aligned}
\end{equation*}
For $I_4$, note that $\Gamma_y(b(s),s,b(\tau),\tau)$ and $M(\tau)$ are both smooth on $[\frac s2, s-s^7]$, we may use the integration by part and have

\begin{equation*}
\begin{aligned}
|I_4|\le&  \left|\Gamma_y(b(s),s,b(s-s^7),s-s^7)\cdot M(s-s^7)\right|+\left|\Gamma_y(b(s),s,b(\frac s2),\frac s2)\cdot M(\frac s2)\right|\\
+&\left|\int_{\frac s2}^{s-s^7} \partial_\tau \Gamma_y(b(s),s,b(\tau),\tau)M'(\tau)d\tau\right|.
\end{aligned}
\end{equation*}
where all the terms are small since $|M(\tau)|$ is much less than any polynomial of $\tau$ and thus $I_4=o(1)$. For $I_5$, recall that $|\Gamma_y(b(s),s,b(\tau),\tau)|\le \frac{C}{\sqrt{s-\tau}}$ and $|M'(\tau)|\le Cs^{-3}$ on $[s-s^7,s]$, we have
$$
|I_5|\le s^{-3}\int_{s-s^7}^s \frac{C}{\sqrt{s-\tau}}d\tau\le C\sqrt s=o(1).
$$
which derives $\lim_{\Delta s\to 0}\left|I_{3,2}\right|=o(1)$ and thus $\lim_{\Delta s\to 0}\left|I_3\right|=o(1)$. Combining \eqref{first term}, \eqref{second term} and \eqref{third term}, we got $\lim_{s\to 0^+}J'_2(s)=0$ and then $\lim_{s\to 0^+}M'(s)=0$, which together with \eqref{relationship of density} derive $\lim_{t\to 0^+}f'_{T_1}(t)=0$ and $f_{T_1}(t)\in C^1[0,T]$.

\end{proof}
%\begin{equation*}
%\left\{
%\begin{aligned}
%&\frac{\partial f_{0}}{\partial t}-\frac{\partial}{\partial x} (xf_{0})-\frac{\partial^{2} f_{0}}{\partial x^{2}}=0,x\in(-\infty,1)\\
%&f_{0}(1,t)=0,f_{0}(-\infty,t)=0,f_{0}(x,0)=\delta(x)
%\end{aligned}
%\right. \label{eq:f0}
%\end{equation*}

%\begin{remark}
%1-3 should follow directly from the regularity of $f_{0}$. The proof can be found in .
%\end{remark}

Next, we can do the first iteration.
\begin{proposition}
\label{prop first iteration}
Let $f_1(x,t)$ be the density function of the measure induced by $F_1(\cdot,t)$ defined in \eqref{iteration_cdf}, it satisfies the following initial condition and the recursive relation
\begin{equation}
\begin{aligned}
\label{first iteration}
f_{1}(x,0)&=0, \quad \forall x\in (-\infty,1),\\
f_{1}(x,t)&=\int^{t}_{0}f_{0}(x,t-s)f_{T_{1}}(s)ds,\quad \forall x\in(-\infty,0)\cup(0,1), t>0.
\end{aligned}
\end{equation}
For any fixed $T>0$, we have
\begin{itemize}
    \item[(\romannumeral1)]
    $f_1(x,t)$ is the classical solution of the following PDE on $(-\infty, 1]\times[0,T]$:
    \begin{numcases}{}
    \frac{\partial f_{1}}{\partial t}-\frac{\partial}{\partial x}(xf_1)-\frac{\partial^2}{\partial x^2}f_1=0,\quad x\in(-\infty,0)\cup (0,1), t\in(0,T], \label{1}\\
    f_1(0^{-},t)=f_{1}(0^{+},t), \quad \frac{\partial}{\partial x} f_{1}(0^{-},t)-\frac{\partial}{\partial x}f_{1}(0^{+},t)=f_{T_1}(t), \quad t\in(0,T], \label{2}\\
    f_{1}(-\infty,t)=0, \quad f_1(1,t)=0, \quad t\in[0,T],\label{3} \\
    f_1(x,0)=0,\quad x\in(-\infty,1)\label{4}
    \end{numcases}
    with
    \beq
    \label{decay 1}
    \lim_{x \to -\infty }\partial_xf_1(x, t)=0, \quad t\in [0,T].
    \eeq
    \item[(\romannumeral2)]
    There is a big enough constant $C_T$ depending only on $T$ such that
    \beq
    |f_{1}(x,t)|\le C_T,\quad  \forall x\in(-\infty,0)\cup(0,1), t \in [0,T],
    \label{bound}
    \eeq
    \beq
    \left|\frac{\partial}{\partial x}f_{1}(x,t)\right|\le C_T, \quad \forall x\in(-\infty,0)\cup(0,1), t \in [0,T].
    \label{bound derivative}
    \eeq
    And at the domain boundary:
    \beq \label{bdf1}
    \left|\frac{\partial}{\partial x}f_{1}(0^-,t)\right| \le C_T, \quad \left|\frac{\partial}{\partial x}f_{1}(0^+,t)\right| \le C_T,\quad \left|\frac{\partial}{\partial x}f_{1}(1^-,t)\right| \le C_T,\quad t \in [0,T].
    \eeq
    \item[(\romannumeral3)]
    For $t>0$, recalling that the density of the second jumping time
    \begin{equation}
    \label{convolution second jump}
    f_{T_{2}}(t)=\int_{0}^{t}f_{T_{1}}(t-s)f_{T_{1}}(s)ds,
    \end{equation}
    we have
    \begin{equation}
    -\frac{\partial f_{1}}{\partial x}(1,t)=f_{T_2}(t).
    \label{first flux}
    \end{equation}
\end{itemize}
\end{proposition}

\begin{proof}
By \eqref{CDF_iterate} and the Fubini formula, we immediately get \eqref{first iteration}. As we have already known $f_0(x,t)$ satisfies PDE \eqref{classical f0}, thus from iteration relationship \eqref{first iteration} and the regularities for $f_0(x,t)$ in Proposition \ref{regularity f0}, we can check that $f_1(x,t)$ satisfies PDE \eqref{classical fn} with $n=1$ and the estimations for $f_1(x,t)$ are valid.

To prove (\romannumeral1), by the regularities of $f_0$ in Proposition \ref{regularity f0}, we have $\forall x\in(-\infty,0)\cup (0,1)$,
\begin{equation*}
\frac{\partial}{\partial x}f_1(x,t)=\int^{t}_{0}\frac{\partial}{\partial x}f_{0}(x,t-s)f_{T_{1}}(s)ds \quad \text{and} \quad  \frac{\partial^2}{\partial x^2}f_1(x,t)=\int^{t}_{0}\frac{\partial^2}{\partial x^2}f_{0}(x,t-s)f_{T_{1}}(s)ds,
\end{equation*}
which together with the decay property \eqref{decay 0} for $f_0$ derive \eqref{decay 1}.\\
Moreover,
\begin{equation*}
\begin{aligned}
\frac{\partial}{\partial t}f_1(x,t)&=\frac{\partial}{\partial t}\int^t_0 f_0(x,t-s)f_{T_{1}}(s)ds\\
&=\lim_{\Delta t \to 0}\int^{t}_{0}\frac{f_{0}(x,t+\Delta t-s)-f_{0}(x,t-s)}{\Delta t}f_{T_{1}}(s)ds\\
&+\lim_{\Delta t \to 0}\frac{\int^{t+\Delta t}_{t}f_{0}(x,t+\Delta t-s)f_{T_{1}}(s)ds}{\Delta t}\\
&=\int_0^t \frac{\partial}{\partial t} f_0(x,t-s) f_{T_{1}}(s) ds.
\label{diff t}
\end{aligned}
\end{equation*}
Thus we have checked \eqref{1} and also got the continuity of $\frac{\partial^2}{\partial x^2}f_1(x,t)$ and $\frac{\partial}{\partial t}f_1(x,t)$. At the same time, \eqref{3} and \eqref{4} are obvious because of the boundary conditions of $f_{0}$ and the formula \eqref{first iteration}. So for the rest of the proof we concentrate on verifying \eqref{2}, which is composed of
\begin{equation}
f_{1}(0^{-},t)=f_{1}(0^{+},t),\quad t\in(0,T] \label{bc 1}
\end{equation}
and
\begin{equation}
\frac{\partial}{\partial x} f_{1}(0^-,t)-\frac{\partial}{\partial x} f_{1}\left(0^+, t\right)=f_{T_{1}}(t),\quad t\in(0,T]. \label{bc 2}
\end{equation}

To show \eqref{bc 1}, note that $\int^{1}_{0}\frac{1}{\sqrt {1-e^{-2s}}}ds <\infty$ and thus for any $\varepsilon >0$, $\exists 0<\delta<t$ s.t. $\int^{\delta}_{0}\frac{1}{\sqrt {1-e^{-2s}}}ds < \frac{\varepsilon}{c}$, where the constant $c$ is the same as in \eqref{density}. With \eqref{first iteration}, we have for any $x\ne 0$,
\begin{equation}
f_{1}\left(x,t\right)=\int_{0}^{t} f_{0}\left(x,t- s\right)f_{T_{1}}(s)ds=\int_{0}^{t-\delta} f_0(x, t-s) f_{T_1}(s) ds+\int_{t-\delta}^{t} f_0(x, t-s) f_{T_1}(s) ds.
\end{equation}
For the 2nd term above, using \eqref{density}:
\begin{equation*}
\int_{t-\delta}^{t} f_0(x, t-s) f_{T_1}(s) ds \le\big\|f_{T_1}\big\|_{L^\infty [0, t]}\\
\int_{0}^{\delta} \frac{c}{\sqrt {1-e^{-2s}}} ds \leq\big\|f_{T_{1}}\big\|_{L^\infty [0,t]} \cdot \varepsilon.
\end{equation*}
While for the first term, one may use \eqref{regularity 3} and see that
\begin{equation*}
\lim_{x_1\to 0^{+},x_2\to 0^{+}} \int_{0}^{t-\delta}\left|f_{0}(x_1,t-s)-f_{0}(-x_2, t-s)\right| f_{T_{1}}(s) ds=0.
\end{equation*}
Since $\varepsilon$ is arbitrary, we get \eqref{bc 1}.

Now to prove  \eqref{bc 2}: noting that
$$ \frac{\partial}{\partial x} f_{1}\left(x_1,t\right)=\int_{0}^{t} \frac{\partial}{\partial x} f_0(x_1,t-s) f_{T_{1}}(s) d s, \quad x_1\in (0,1)
$$
and for any $t-s\ne 0$,
\begin{equation*}
\begin{aligned}
\frac{\partial}{\partial x} f_0(x_1, t-s)&=-\int_{x_1}^{1} \frac{\partial^{2}}{\partial x^2} f_0(x, t-s)dx+\frac{\partial}{\partial x} f_0(1, t-s)\\
&=-\int_{x_1}^{1}\left[\frac{\partial f_{0}}{\partial t}(x,t-s)-\frac{\partial}{\partial x}(x f_0)(x,t-s)\right] dx+\frac{\partial}{\partial x} f_0(1,t-s)\\
&=\int_{x_1}^{1}\left[\frac{\partial}{\partial x}(x f_0)(x, t-s)-\frac{\partial f_0}{\partial t}(y,t-s)\right]dx+\frac{\partial}{\partial x} f_0(1,t-s)\\
&=f_0(1, t-s)-x_1 f_0(x_1, t-s)-\int_{x_1}^{1} \frac{\partial f_0}{\partial t}(x, t-s) dx+\frac{\partial}{\partial x} f_0(1, t-s)\\
&=-x_1 f_0(x_1, t-s)-\int_{x_1}^{1} \frac{\partial f_{0}}{\partial t}(x, t-s) dx-f_{T_1}(t-s).
\end{aligned}
\end{equation*}
Thus:
\begin{equation}
\begin{aligned}
&\frac{\partial}{\partial x} f_1(x_1,t)=\int_{0}^{t} \frac{\partial}{\partial x} f_0(x_1,t-s) f_{T_1}(s) ds\\
&=\int_{0}^{t}\left[-x_1 f_0(x_1, t-s)-\int_{x_1}^{1} \frac{\partial f_{0}}{\partial t}(x, t-s) dx-f_{T_1}(t-s)\right] f_{T_{1}}(s) ds.
\label{delta x}
\end{aligned}
\end{equation}
Similarly for any $x_2>0$, we have
\begin{equation*}
\begin{aligned}
&\frac{\partial}{\partial x} f_0(-x_2, t-s)\\
=&\int_{-\infty}^{-x_2} \frac{\partial^2}{\partial x^2} f_0(x,t-s) dx +0\\
=&\int_{-\infty}^{-x_2}\left[\frac{\partial f_{0}}{\partial t}(x, t-s)-\frac{\partial}{\partial z}(xf_0(x, t-s))\right] dx\\
=&x_2 f_0(-x_2,t-s)+\int_{-\infty}^{-x_2} \frac{\partial f_0}{\partial t}(x, t-s) dx.
\end{aligned}
\end{equation*}
And thus
\begin{equation}
\label{delta y}
\frac{\partial}{\partial x} f_1(-x_2, t)=\int_0^t\left[x_2 f_0(-x_2, t-s)+\int_{-\infty}^{-x_2} \frac{\partial f_0}{\partial t}(x, t-s) dx \right]  f_{T_1}(s) ds.
\end{equation}

Combining \eqref{delta x} and \eqref{delta y}, we have for all $x_1\in (0,1)$ and $x_2 >0$,
\begin{equation}
\begin{aligned}
\label{x y}
&\frac{\partial}{\partial x} f_1(-x_2, t)-\frac{\partial}{\partial x} f_1(x_1, t)\\
=&x_2 \int_{0}^{t} f_{0}(-x_2, t-s) f_{T_1}(s) d s+x_1 \int_{0}^{t} f_{0}(x_1, t-s) f_{T_{1}}(s) d s\\
+&\int_{0}^{t}\left[\int_{\BR\setminus[-x_2, x_1]} \frac{\partial f_0}{\partial t}(x, t-s) dx+f_{T_1}(t-s)\right] f_{T_1}(s) ds\\
=&: I_6 +I_7+I_8.
\end{aligned}
\end{equation}
For $I_6$, we have by \eqref{density}:
\begin{equation*}
I_6 \le \|f_{T_{1}}\|_{L^\infty[0, t]} \cdot x_2 \cdot \int_{0}^{t} \frac{c}{\sqrt {1-e^{-2s}}} ds \to 0 \quad as \quad  x_2 \to 0^+.
\end{equation*}
And $I_7\to 0$ by the same argument, it now suffices to show
\begin{equation}
I_8 \to f_{T_1}(t) \quad as \quad x_1,x_2 \to 0^+.
\label{end}
\end{equation}
In the rest of our calculations, integrand of  $I_8$ will be called $H(s)$. As a result of Proposition \ref{regularity [0,T]}, for any $\varepsilon >0$, we let the chosen $\delta$ small enough such that
\begin{align}
&\quad \delta \left\| f_{T_1}\right\|^2_{L^\infty [0,t]}<\varepsilon\label{assumption 1}\\
&\quad \int^{t_2}_{t_1}\left|f^{\prime}_{T_1}(s)\right| ds<\varepsilon, \quad \forall t_1<t_2<t,\ t_2 - t_1 <\delta\label{assumption 2}\\
&\quad P(T_1<\delta)<\varepsilon. \label{assumption 3}
\end{align}
Then for the fixed $\delta >0$ defined above,
\begin{equation}
I_8=\int_{0}^{t-\delta} H(s) d s+\int_{t-\delta}^{t} H(s) ds=:I_{8,1}+I_{8,2}.
\label{end 2}
\end{equation}
For $I_{8,1}$, we have by \eqref{regularity 3} and \eqref{T_1 pdf},
\begin{equation*}
\begin{aligned}
\left|I_{8,1}\right|=&\left| \int_{0}^{t-\delta}\left[\int_{R \setminus [-x_2, x_1]} \frac{\partial f_0}{\partial t}(y, t-s) dy+f_{T_1}(t-s)\right] f_{T_1}(s) ds\right|\\
=&\left| \int_{0}^{t-\delta}\int_{-x_2}^{x_1} \frac{\partial f_0}{\partial t}(y, t-s)f_{T_1}(s) dy ds\right|\\
\le& \|f_{T_1}\|_{L^\infty [0, t]} \int_{0}^{t-\delta} \int_{-x_2}^{x_1}\left\|\frac{\partial f_0}{\partial t}\right\|_{L^\infty (-\infty, 1] \times [\delta,T]} dy ds\\
\le& t \cdot \|f_{T_1}\|_{L^\infty[0, T]} \cdot \left\| \frac{\partial f_{0}}{\partial t}\right\|_{L^\infty(-\infty, 1] \times[\delta, T]}\cdot(x_1+x_2)
\end{aligned}
\end{equation*}
which $\to 0$ as $x_1, x_2 \to 0$. As for $I_{8,2}$,
\begin{equation*}
I_{8,2}=\int_{t-\delta}^{t}\left[\int_{\BR \setminus [-x_2, x_1 ]} \frac{\partial f_0}{\partial t}(y, t-s) d y+f_{T_{1}}(t-s)\right] f_{T_{1}}(s) ds.
\end{equation*}
One may first see by \eqref{assumption 1}, we have $\int_{t-\delta}^{t} f_{T_{1}}\left(t-s\right) f_{T_{1}}(s) d s \leq \varepsilon$. Moreover, for any $x_1,x_2 >0$, note that function $\frac{\partial f_{0}}{\partial t}(y,t-s)f_{T_1}(s)$ is bounded and continuous on the region $(R\setminus[-x_2,x_1])\times [t-\delta,t]$. One may apply Fubini's formula and have:
\begin{equation}
\label{end2.2}
I_{8,2}=\int_{R \setminus [-x_2, x_1]} \int_{t-\delta}^{t} \frac{\partial f_0}{\partial t}(y, t-s) f_{T_{1}}(s) ds dy.
\end{equation}
At the same time, by \eqref{Gest} we have for any fixed $t>0, y\notin [-x_2,x_1]$,
\begin{equation*}
f_0(y, t- s) f_{T_1}(s) \in C^1[t-\delta , t)
\end{equation*}
and
\begin{equation*}
\lim _{s \rightarrow t^-}  f_0(y, t-s) f_{T_1}(s)=0.
\end{equation*}
Thus, one may apply integration by parts and have
\begin{equation}
\begin{aligned}
\label{end2.3}
&\int_{t-\delta}^{t} \frac{\partial f_0}{\partial t}(y,t-s) f_{T_1}(s) ds\\
=&\left(-f_0(y,t-s) f_{T_1}(s)\right)\Big|_{t-\delta} ^{t}+\int_{t-\delta}^{t} f_0(y, t-s) f_{T_1}^{\prime}(s) ds\\
=&f_0(y, \delta)f_{T_1}(t-\delta)+\int_{t-\delta}^{t}f_0(y,t-s)f_{T_1}^{\prime}(s) ds.
\end{aligned}
\end{equation}
Plugging \eqref{end2.3} back to \eqref{end2.2} and applying the Fubini theorem once again, we have
\begin{equation}
\begin{aligned}
I_{8,2}&=\left[\int_{R \setminus[-x_2, x_1]} f_0(y, \delta) dy \right]f_{T_1}(t-\delta)+ \int_{t-\delta}^{t}\int_{R \setminus[-x_2, x_1]}f_0(y,t-s)dy f_{T_1}^{\prime}(s) ds\\
&=:I_{9}+I_{10}.
\label{end2.4}
\end{aligned}
\end{equation}
First for $I_{10}$, noting that $f_0$ is a p.d.f., for any $s \in (t-\delta, t)$ we have
\begin{equation*}
\int_{\BR \setminus[-x_2, x_1]}f_0(y,t-s)dy\le 1,
\end{equation*}
which together with \eqref{assumption 2} derive
\begin{equation}
\left|I_{10} \right|\le \int_{t-\delta}^{t} \left|f_{T_1}^{\prime}(s)\right| ds <\varepsilon.
\end{equation}
Then for $I_9$, by \eqref{assumption 3} we have
\begin{equation*}
\begin{aligned}
\lim_{x_1 \to 0^+, x_2 \to 0^+}\int_{R \setminus[-x_2, x_1]} f_0(y, \delta) dy=\int_{-\infty}^{1}f_0(y, \delta)dy=P(T_1>\delta)\in [1-\varepsilon, 1]
\end{aligned}
\end{equation*}
and
\begin{equation*}
\begin{aligned}
\big|f_{T_1}(t-\delta)-f_{T_1}(t)\big|<\varepsilon.
\end{aligned}
\end{equation*}
Thus we have for all sufficiently small $x_1>0, x_2>0$,
\begin{equation}
\begin{aligned}
\left|I_9-f_{T_1}(t)\right|=\left|I_9-f_{T_1}(t-\delta)\right|+\left|f_{T_1}(t-\delta)-f_{T_1}(t) \right|<(\|f_{T_1}\|_{L^{\infty}[0,t]}+1)\varepsilon. \label{end2.5}
\end{aligned}
\end{equation}
Now combing from \eqref{end 2} to \eqref{end2.5}, we have concluded that $I_8 \to f_{T_1}(t)\ \text{as} \ x_1,x_2 \to 0^+$, which together with $I_6\to 0$ and $I_7\to 0$ derive \eqref{bc 2}.

As for (\romannumeral2), we first derive \eqref{bound} and \eqref{bound derivative} that are essential in getting \eqref{first flux} and also set the basis for subsequent iterations. First we verify \eqref{bound} and without loss of generality, one may assume $T>1$. So when $t\in(0,T]$ and $x\in (-\infty,0)\cup(0,1)$,
\begin{equation*}
\begin{aligned}
f_{1}(x, t) &=\int_{0}^{t} f_{0}(x, t-s) f_{T_{1}}(s) ds \\ & \leqslant \int_{0}^{t-1} f_{0}(x, t-s) f_{T_{1}}(s) d s+\|f_{T_{1}}\|_{L^\infty[0, T]}\int_{0}^{1} \frac{1}{\sqrt{1-e^{-2s}}}ds \\&
\leqslant\left\|f_{0}(x, t)\right\|_{L^{\infty}(-\infty, 1] \times[1,\infty)}+C'_T=C_T.
\end{aligned}
\end{equation*}
And for \eqref{bound derivative}, without loss of generality, one may assume that $x>0$ and by \eqref{delta x} we have:
\begin{equation*}
\frac{\partial f_{1}}{\partial x}(x, t)=\int_{0}^{t}\left[-x f_{0}(x, t-s)-\int_{x}^{1} \frac{\partial f_{0}}{\partial t}(y, t-s) d y-f_{T_{1}}(t-s)\right] f_{T_{1}}(s) ds=: I_{11}+I_{12}+I_{13}.
\end{equation*}
Using the estimate \eqref{decay at infinity} for $f_0(x,t)$, one have
\begin{equation*}
\left\{\begin{array}{c}{\left|I_{11}\right|\le C \cdot F_{T_{1}}(t) \le C_{T}} \\ \left|I_{13} \right| \le\|f_{T_{1}}\|_{L^\infty[0, T]} F_{T_{1}}(t)\le C_T. \end{array}\right.
\end{equation*}
For the remaining $I_{12}$, formula twice and integration by parts together with the fact that $f_0(\cdot, t)$ is a p.d.f. to have:
\begin{equation*}
\begin{aligned}
\left|I_{12}\right|
=&\left|\int_{x}^{1} \int_{0}^{t} \frac{\partial f_{0}}{\partial t}(y, t-s) f_{T_{1}}(s) d s d y \right|\\
=&\left|\int_{x}^{1} \int_{0}^{t} f_{0}(y, t-s) f_{T_{1}}^{\prime}(s) d s d y\right|\\
=&\left|\int_{0}^{t} \int_{x}^{1} f_{0}(y, t-s) dy f_{T_{1}}^{\prime}(s) d s \right|\\ \le & \int_{0}^{t}\left|f_{T_{1}}^{\prime}(s)\right| ds \\  \le &C_{T}.
\end{aligned}
\end{equation*}
Because of the proof of \eqref{2} in (\romannumeral1), property (\romannumeral2) of Proposition \ref{regularity f0} and representation
\eqref{first iteration}, we know that $\frac{\partial}{\partial x}f_{1}(0^-,t)$, $\frac{\partial}{\partial x}f_{1}(0^+,t)$ and $\frac{\partial}{\partial x}f_{1}(1^-,t)$ are well-defined, and thus by taking the one side limit in \eqref{bound derivative}, we immediately get \eqref{bound derivative} and thus we complete the proof of (\romannumeral2).

Finally for (\romannumeral3), using integral representation \eqref{jump time convolution}, we immediately get \eqref{convolution second jump}. Recalling that $f_1(1,t)=0$, $\forall t>0$, it suffices to prove
\begin{equation}
\lim _{x_1 \rightarrow 0^{+}} \frac{f_{1}(1-x_1, t)}{x_1} =f_{T_{2}}(t) ,\quad \forall t>0.
\label{N1}
\end{equation}
Now note that for all $0<x_1<\frac{1}{2}$,
\begin{equation*}
f_{1}(1-x_1, t)=\int_{0}^{t} f_{0}(1-x_1, t-s) f_{T_1}(s) d s
\end{equation*}
While at the same time by mean value theorem on $f_0$, for all $s\in [0,t]$, $\exists \ \xi_{t-s}(x_1)\in [1-x_1,1] \subset [\frac{1}{2},1]\quad s.t.$
\begin{equation*}
\begin{aligned}\\ \frac{f_0(1-x_1, t-s)}{x_1} &=-\frac{f_{0}(1, t-s)-f_{0}(1-x_1, t-s)}{x_1} \\ &=-\frac{\partial}{\partial x} f_{0}\left(\xi_{t-s}(x_1), t-s\right). \end{aligned}
\end{equation*}
Note for all $0<x_1<\frac{1}{2}$, by (\romannumeral2) of Proposition \ref{regularity f0} for $f_0$:
\begin{equation*}
\frac{\partial}{\partial x} f_{0}(\xi_{t-s}(x_1), t-s) \leqslant\left\|\frac{\partial f_0}{\partial x}\right\|_{L^{\infty}_{\left[\frac{1}{2}, 1\right] \times[0,T]}}
\end{equation*}
and we have
\begin{equation*}
{\lim _{x_1 \rightarrow 0^{+}} \frac{\partial}{\partial x} f_{0}\left(\xi_{t-s} (x_1), t-s\right)=\frac{\partial}{\partial x} f_{0}(1, t-s)}.
\end{equation*}
By the dominated convergence theorem,
\begin{equation*}
{\lim _{x_1 \rightarrow 0^{+}} \frac{f_{1}(1-x_1, t)}{x_1}=-\int_{0}^{t} \frac{\partial}{\partial x} f_{0}(1, t-s) f_{T_1}(s) d s} \\ {=\int_{0}^{t} f_{T_{1}}(t-s) f_{T_{1}}(s) d s} \\ {=f_{T_2}(t)}.
\end{equation*}
and thus
$$
-\frac{\partial f_{1}}{\partial x}(1,t)=f_{T_2}(t).
$$
The proof of Proposition \ref{prop first iteration} is complete.

\end{proof}

%At the point, we are ready to state and prove the main theorem on classic solutions via iterative construction.

Similarly by \eqref{CDF_iterate}, for all $n \ge 1$, we have
\begin{equation}
\begin{aligned}
\label{nth iteration}
f_n(x,0)&=0, \quad \forall x\in (\infty,1),\\
f_{n}(x,t)&=\int^{t}_{0}f_{n-1}(x,t-s)f_{T_{1}}(s)ds,\quad \forall x\in(-\infty,0)\cup(0,1),t>0
\end{aligned}
\end{equation}
%\begin{equation*}
%\left\{\begin{array}{l}{f_{n}(x, t)=\int_{0}^{t} f_{n-1}(x, t) f_{T_1}(s) d s} \\ {f_{n}(x, 0)=0}\end{array}\right.
%\end{equation*}
and
\begin{equation*}
f_{T_{n+1}}(t)=\int_{0}^{t} f_{T_{n}}(t-s) f_{T_{1}}(s) ds.
\end{equation*}
Hence, the iterative construction is feasible, and we can show

\begin{proposition}
\label{main}
For each $n \ge 1$, let $f_n(x,t)$ be the density function of the measure induced by $F_n(\cdot,t)$ defined in \eqref{iteration_cdf}. For any fixed $T>0$, we have

\begin{itemize}
    \item[(\romannumeral1)]
    $f_n$ is the classic solution of the following PDE:
    \begin{numcases}{}
    \frac{\partial f_{n}}{\partial t}-\frac{\partial}{\partial x}\left(x f_{n}\right)-\frac{\partial^{2}}{\partial x^{2}} f_{n}=0, \quad x \in(-\infty, 0) \cup(0,1),t\in[0,T],\label{5}\\
    f_{n}(0^-, t)=f_{n}\left(0^{+}, t\right), \quad \frac{\partial}{\partial x} f_{n}\left(0^{-}, t\right)-\frac{\partial}{\partial x} f_{n}\left(0^{+}, t\right)=f_{T_n}(t),\quad t\in(0,T].\label{6}\\
    f_{n}(-\infty, t)=0, \quad f_{n}(1, t)=0,\quad t\in[0,T]\label{7}\\
    f_n(x,0)=0,\quad x \in(-\infty, 1)\label{8}
    \end{numcases}
    with
    \beq
    \label{decay n}
    \lim_{x \to -\infty }\partial_xf_n(x, t)=0, \quad t\in [0,T].
    \eeq
    \item[(\romannumeral2)]
    There is a $C_T$ that depends only on $T$ such that
    \beq\label{bound fn}
    |f_n(x,t)|\le C_T,\quad \forall x\in (-\infty,0)\cup(0,1), t\in[0,T],
    \eeq
    \beq\label{bound derivative_fn}
    \left|\frac{\partial}{\partial x}f_{n}(x,t)\right|\le C_T, \quad \forall x\in(-\infty,0)\cup(0,1), t \in [0,T],
    \eeq
    and at the domain boundary
    \beq \label{bdfn}
    \left|\frac{\partial}{\partial x}f_{n}(0^-,t)\right| \le C_T, \quad \left|\frac{\partial}{\partial x}f_{n}(0^+,t)\right| \le C_T,\quad \left|\frac{\partial}{\partial x}f_{n}(1^-,t)\right| \le C_T.
    \eeq
    \item[(\romannumeral3)]
    For $t>0$, $f_n$ is differentiable at $x=1$ and
    \begin{equation}
    \label{T_n pdf}
    -\frac{\partial f_n}{\partial x}(1,t)=f_{T_{n+1}}(t).
    \end{equation}
\end{itemize}
\end{proposition}
\begin{proof}

The proof of Proposition \ref{main} follows from induction. By Proposition \ref{prop first iteration}, we have presented the inductive basis at $n=1$. Now assuming the inductive hypothesis holds up to $n > 1$, To prove (\romannumeral1), by
\begin{equation*}
f_{n+1}(x,t)=\int^t_0 f_n(x,t-s)f_{T_{1}}(s)ds,
\end{equation*}
one may immediately see \eqref{5},\eqref{7},\eqref{8} and \eqref{decay n} hold. For \eqref{6}, note that $f_n(0^-,t)=f_n(0^+,t)$, $\forall t>0$, and that $\left|f_n(x,t)\right|\le C_T$, $\forall x\in (-\infty,0)\cup(0,1), t\le T$. By the dominated convergence theorem, we have
\begin{equation*}
\begin{aligned}
&\lim_{x_1 \to 0^+,x_2 \to 0^+}|f_{n+1}(x_1,t)-f_{n+1}(-x_2,t)|\\
\le& \lim_{x_1 \to 0^+,x_2 \to 0^+} \int^t_0|f_n(x_1,t-s)-f_n(-x_2,t-s)| f_{T_1}(s)ds\\
=&0.
\end{aligned}
\end{equation*}
So we have
\begin{equation*}
f_{n+1}(0^-,t)=f_{n+1} (0^+,t).
\end{equation*}
Similarly,
\begin{equation*}
\begin{aligned}
&\lim_{x_1 \to 0^+,x_2 \to 0^+} \frac{\partial f_{n+1}}{\partial x}(x_1,t)-\frac{\partial f_{n+1}}{\partial x}(-x_2,t) \\
=&\lim_{x_1 \to 0^+,x_2 \to 0^+} \int^t_0 \frac{\partial f_{n}}{\partial x}(x_1,t-s)-\frac{\partial f_{n}}{\partial x}(-x_2,t-s) f_{T_1}(s)ds.
\end{aligned}
\end{equation*}
By the inductive hypothesis
%\begin{equation*}
%\frac{\partial}{\partial x}f_n(0^-,t)-\frac{\partial}{\partial x}f_n(0^+,t)=f_{T_n}(t)
%\end{equation*}
%and
%\begin{equation*}
%\left|\frac{\partial}{\partial x}f_n(y,t)\right|\le C_T, \quad \forall 0<|y|\le 1,\forall t\le T
%\end{equation*}
and the dominated convergence theorem, we have
\begin{equation*}
\frac{\partial}{\partial x} f_{n+1}\left(0^{-}, t\right)-\frac{\partial}{\partial x} f_{n+1}\left(0^{+}, t\right)=f_{T_{n+1}}(t).
\end{equation*}
As for (\romannumeral2), to check the additional regularity conditions, note that inductive hypothesis,
\begin{equation*}
0\le f_{n+1}(x,t)=\int^t_0 f_n(x,t-s)f_{T_1}(s)ds\le C_T.
\end{equation*}
And for any $y \in (-\infty,0)\cup(0,1)$ and $t\le T$,
\begin{equation*}
\left|\frac{\partial f_{n+1}}{\partial x}(y,t)\right|\le \int^t_0\left|\frac{\partial f_n}{\partial x}(y,t-s)\right|f_{T_1}(s)ds \le C_T.
\end{equation*}
Using similar arguments as in Proposition \ref{prop first iteration}, we have $\frac{\partial}{\partial x} f_{n+1}\left(0^{-}, t\right)$, $\frac{\partial}{\partial x} f_{n+1}\left(0^{+}, t\right)$ and $\frac{\partial}{\partial x} f_{n+1}\left(1^{-}, t\right)$ are individually bounded  by $C_T$.
Finally for (\romannumeral3), noting that $\left|\frac{\partial f_n}{\partial x}(y, t)\right|\le C_T$ for all $t\le T$, $0<y<1$, the proof of
\begin{equation*}
-\frac{\partial f_n}{\partial x}(1,t)=f_{T_{n+1}}(t),\quad \forall t>0.
\end{equation*}
follows from the same treatment as in Proposition \ref{prop first iteration}.
%And by iteration equality \eqref{nth iteration}, the decay of $\frac{\partial}{\partial x}f_{n+1}(x, t)$ when $x \to -\infty$ and the continuity of $\frac{\partial^2}{\partial x^2}f_{n+1}(x,t)$ and $\frac{\partial}{\partial t}f_{n+1}(x,t)$ is obvious. Thus the proof of Proposition \ref{main} is completed.

\end{proof}

Now we can finish the proof of Theorem \ref{thm_forward}.

\begin{proof}
[Proof of Theorem \ref{thm_forward}:]
Based on the previous analysis in Proposition \ref{regularity f0}-\ref{main}, we have shown that for $n \ge 0$, $f_n$ is the density function of the measure induced by $F_n(\cdot,t)$ defined in \eqref{iteration_cdf} as well as the solution to the sub PDE problems \eqref{classical f0} and \eqref{classical fn}. Next, we consider the density function of the stochastic process $X_t$ as in \eqref{process definition} that admits the series representation $f(x, t)=\sum_{n=0}^{+\infty}f_n(x, t)$.

In order to prove that $f(x,t)$ satisfies the properties in Theorem \ref{thm_forward}, we first show that the relevant derivatives of $f(x,t)$ also have the series representations and the series converge uniformly so that we can pass the regularity from $f_n(x,t)$ to $f(x,t)$. Besides, noting that $f_n$ is the solution to the sub PDE problems \eqref{classical f0} and \eqref{classical fn}, and thus we can show in the following $f=\sum_{n=0}^{+\infty}f_n$ satisfies the \eqref{eq_planck}, which is the summation of sub PDE problems PDE \eqref{classical f0} and \eqref{classical fn}.

For any fixed $T>0$, we first show the uniform convergence of the relevant derivatives of $\sum_{n=0}^{+\infty}f_n(x,t)$ on $((-\infty,0) \cup (0,1])\times [0,T]$. By \eqref{nth iteration}, $\forall x_0 \in (-\infty,0) \cup (0,1]$, we have for any $0\le t \le T$ and $n\ge 1$
\beq
\label{iteration bound}
\left|\frac{\partial}{\partial x}f_n(x_0,t)\right|\le \int ^{T}_{0}f_{T_{1}}(s)ds \cdot \max_{t\in [0,T]}\left|\frac{\partial}{\partial x}f_{n-1}(x_0,t)\right| \le \rho_T \max_{t\in [0,T]}\left|\frac{\partial}{\partial x}f_{n-1}(x_0,t)\right|,
\eeq
%\beq
%\max_{0<t\le T}\frac{\partial^2}{\partial %x^2}f_n(x_0,t)=\int ^{T}_{0}f_{T_{1}}(s)ds \cdot %\max_{0<t\le T}\frac{\partial^2}{\partial %x^2}f_{n-1}(x_0,t) \le \rho_T \max_{0<t\le %T}\frac{\partial^2}{\partial x^2}f_{n-1}(x_0,t)
%\eeq
where
\beq
\label{constant<1}
\rho_T=\int ^{T}_{0}f_{T_{1}}(s)ds=\prob_{0}(T_1\le T) \in (0,\,1)
\eeq
is a constant that depends only on $T$. The proof of \eqref{constant<1} is quite standard in probability and thus we put the whole proof of it in Appendix. With \eqref{iteration bound}, we have
\beq
\label{transformation}
\sum_{n=0}^{+\infty}\max_{t\in[0,T]}\left|\frac{\partial}{\partial x}f_n(x_0, t)\right|\le \sum_{n=0}^{+\infty} \rho^n_T  \max_{t\in[0, T]}\left|\frac{\partial}{\partial x}f_{0}(x_0, t)\right|=\frac{1}{1-\rho_T}\max_{t\in [0,T]}\left|\frac{\partial}{\partial x}f_{0}(x_0, t)\right|,
\eeq
which implies to show the uniform convergence of such series, it suffices to check the regularities of $f_0(x,t)$. In fact, with (\romannumeral2) of Proposition \ref{regularity f0}, we know that for any $\varepsilon_0\in(0,1)$, $f_{0}(x,t)\in C^{2, 1} \left(\left((-\infty, -\varepsilon_0]\cup [\varepsilon_0,1]\right)\times [0,T]\right)$ and thus the last term in \eqref{transformation} has a uniform bound on any compact subset of $(-\infty,0) \cup (0,1]$, i.e., for any compact subset $I$ of $(-\infty,0) \cup (0,1]$,
$$
\sum_{n=0}^{+\infty}\max_{t\in[0,T]}\max_{x\in I}\left|\frac{\partial}{\partial x}f_n(x_0, t)\right|\le \frac{1}{1-\rho_T}\max_{t\in [0,T]}\max_{x\in I}\left|\frac{\partial}{\partial x}f_{0}(x_0, t)\right|< +\infty.
$$
%\beq
%\sum_{n=0}^{+\infty}\frac{\partial^2}{\partial x^2}f_n(x_0,t)\le \sum_{n=0}^{+\infty} \rho^n_T  \max_{0<t\le T}\frac{\partial^2}{\partial x^2}f_{0}(x_0, t)=\frac{1}{1-\rho_T}\max_{0<t\le T}\frac{\partial^2}{\partial x^2}f_{0}(x_0, t)
%\eeq
With the same treatment, we know that
\beq
\label{equation series}
\sum_{n=0}^{+\infty}f_n(x,t),\
\sum_{n=0}^{+\infty}\frac{\partial}{\partial t}f_n(x,t),\ \sum_{n=0}^{+\infty}\frac{\partial}{\partial x}( x f_n(x,t) )\  \text{and}\  \sum_{n=0}^{+\infty}\frac{\partial^2}{\partial x^2}f_n(x,t)
\eeq
are inner closed uniformly convergent on $((-\infty,0) \cup (0,1])\times [0,T]$, and thus we can exchange the derivative and the summation in \eqref{equation series}. And by \eqref{transformation}, we have
$$
\max_{t\in[0,T]}\left|\partial_xf(x_0, t)\right|\le\sum_{n=0}^{+\infty}\max_{t\in[0,T]}\left|\frac{\partial}{\partial x}f_n(x_0, t)\right|\le \frac{1}{1-\rho_T}\max_{t\in [0,T]}\left|\frac{\partial}{\partial x}f_{0}(x_0, t)\right|
$$
With the same treatment, we can get the same bounds for the series in \eqref{equation series}, from which we can analyse the regularities of $f(x,t)$ by estimating $f_0(x,t)$.

To check (\romannumeral1), we show that $N(t)=-\frac{\partial}{\partial x}f(1^-,t)$ is well-defined and $N(t)$ has a series representation in terms of the densities of jumping times. In fact, by uniform convergence, it is clear that $\sum_{n=0}^{+\infty}\frac{\partial}{\partial x}f_n(1^-,t)$ uniformly converges on $[0,T]$. In particular,
$$
\frac{\partial f}{\partial x}(1^-,t)=\sum_{n=0}^{\infty}\frac{\partial f_n}{\partial x}(1^-,t).
$$
Then by \eqref{T_1 pdf} and \eqref{T_n pdf}, we also have
\beq
N(t)=-\frac{\partial f}{\partial x}(1^-,t)=\sum_{n=0}^{\infty} f_{T_{n}}(t).
\eeq
Note that $f_{T_{n}}(t) \in C[0,T]$, and thus $N(t)\in C[0,T]$. Hence, (\romannumeral1) is completely proved.

With the uniform convergence of the series representations and the regularities of $f_0(x,t)$ in Proposition \ref{regularity f0}, we can easily show (\romannumeral2),  (\romannumeral3),  (\romannumeral4),  (\romannumeral5) of Theorem \ref{thm_forward}.
By \eqref{decay 0} and \eqref{decay n}, we have
$$
\lim_{x \to -\infty }\partial_xf(x, t)= \sum_{n=0}^{+\infty}\lim_{x \to -\infty }\partial_xf_n(x, t)=0, \quad t\in (0,T],
$$
and thus (\romannumeral5) is valid. Similarly, the uniform convergence together with the continuity of $f_n$, $\partial_{xx}f_n$ and $\partial_tf_n$ on $\left((-\infty,0) \cup (0,1)\right)\times (0,T]$ implies (\romannumeral2) and (\romannumeral3).
To check (\romannumeral4), we aim to show that $f_x(0^-, t)$ and $f_x(0^+, t)$ is well-defined for $t\in (0,T]$. With the similar analysis, we can prove that for fixed $0<t\le T$, $\sum_{n=0}^{\infty}\frac{\partial f_n}{\partial x}(x,t)$ uniformly converge on $[-1,0)$ and $(0,1]$, which together with Lemma \ref{Green function estimation} and the existence of one-side limits given in \eqref{regularity 3} and \eqref{bdfn} derives (\romannumeral4) of Theorem \ref{thm_forward}.

%
%\zz{need to show the convergence of  $\sum_{n=0}^{+\infty}\frac{\partial}{\partial x}( x f_n(x,t) )$.}
%
%\zz{Ziheng, please modify proposition 3.1~ 3.3 and the proofs accordingly.}

Finally, to prove (\romannumeral6), that is, the density $f$ satisfies the PDE  problem \eqref{eq_planck}, we need to show that the equation is satisfied as well as all the conditions are met. With uniform convergence, we can sum the equation \eqref{4} from $n=0$ to $+\infty$, and thus for any $(x,t) \in ((-\infty,0) \cup (0,1))\times (0,T]$,
\begin{equation}
\label{sum equation}
\begin{aligned}
&\frac{\partial f}{ \partial t} - \frac{\partial }{\partial x}\left ( x f\right) - \frac{\partial ^2 f}{\partial x^2} \\
=&\frac{\partial }{ \partial t} (\sum_{n=0}^{+\infty}f_n(x, t)) - \frac{\partial }{\partial x}\left ( \sum_{n=0}^{+\infty}xf_n(x, t)\right) - \frac{\partial ^2 }{\partial x^2}(\sum_{n=0}^{+\infty}f_n(x, t))\\
=&\sum_{n=0}^{+\infty} \left(\frac{\partial f_{n}}{\partial t}-\frac{\partial}{\partial x}\left(x f_{n}\right)-\frac{\partial^{2}}{\partial x^{2}} f_{n}\right)\\
=&0.
\end{aligned}
\end{equation}
With the regularities of $f$ proved above, all the initial and boundary conditions in \eqref{eq_planck} are trivially satisfied except that we need to prove the jump condition on $f_x$ at $x=0$. Given any fixed $t>0$, for any $\epsilon>0$, due to the uniform convergence, there is a constant $N<\infty$ such that
%\beq
%\left|\frac{\partial f}{\partial x}(1, t)-\sum_{n=1}^{N}f_{T_{n}}(t)\right|= \left|\sum_{n=N+1}^{\infty}f_{T_n}(t) \right|<\epsilon,
%\eeq
\beq  \label{bound decay}
\left|\sum_{n=N+1}^{\infty}\frac{\partial f_n}{\partial x}(x, t)\right| < \epsilon, \quad \forall x \in (-\infty, 0]\cup [0,1],
\eeq
where at $0,1$ the derivatives are understood in the one-sided sense.
%and
%For any fixed $t>0$ and any $\epsilon>0$, there is a $N<\infty$ such that
%\beq
%\label{bound decay}
%\sum_{n=N+1}^{\infty} \max_{|x|<1} \left|\frac{\partial f_n}{\partial x}(x, t)\right|\le \sum_{n=N+1}^{\infty}C_TP(T_1\le t)^n<\epsilon.
%\eeq
%This immediately implies that
%\beq
%\sum_{n=N+1}^{\infty} \left|\frac{\partial f_n}{\partial x}(1, t)\right|= \sum_{n=N+1}^{\infty}\left|f_{T_{n+1}}(t)\right|<\epsilon.
%\eeq
%and for
%$$
%\frac{\partial f}{\partial x}(1,t)=\sum_{n=0}^{\infty}\frac{\partial f_n}{\partial x}(1,t)
%$$
%we have
%\beq
%\left|\frac{\partial f}{\partial x}(1, t)-\sum_{n=1}^{N}f_{T_{n}}(t)\right| <\epsilon.
%\eeq
Moreover, for the now fixed N, by \eqref{6}, there exists $\delta>0$, such that for all $y<0<x$, $|x|, |y|\le \delta$,
\beq
\label{dist bound}
\left|\frac{\partial f_0}{\partial x}(x, t)-\frac{\partial f_0}{\partial x}(y, t)\right|\le \epsilon
\eeq
and
\beq
\label{dist bound n}
\sum_{n=1}^{N}\left|\frac{\partial f_n}{\partial x}(x, t)-\frac{\partial f_n}{\partial x}(y, t)+f_{T_n}(t) \right|<\epsilon.
\eeq
Combining \eqref{bound decay}-\eqref{dist bound n}, we have
\begin{equation*}
\begin{aligned}
&\left|\frac{\partial f}{\partial x}(x, t)-\frac{\partial f}{\partial x}(y, t)+\sum_{n=1}^{\infty}f_{T_n}(t)\right|\\
=&\left|\sum_{n=0}^{\infty}\left(\frac{\partial f_n}{\partial x}(x, t)-\frac{\partial f_n}{\partial x}(y, t)\right)+\sum_{n=1}^{\infty}f_{T_n}(t)\right|\\
\le&
\left|\frac{\partial f_0}{\partial x}(x, t)-\frac{\partial f_0}{\partial x}(y, t)\right|+ \left|\sum_{n=1}^{N}\left(\frac{\partial f_n}{\partial x}(x, t)-\frac{\partial f_n}{\partial x}(y, t)+f_{T_n}(t)\right) \right| \\
+&\left|\sum_{n=N+1}^{\infty}\frac{\partial f_n}{\partial x}(x, t)\right|+\left|\sum_{n=N+1}^{\infty}\frac{\partial f_n}{\partial x}(y, t)\right|+\left|\sum_{n=N+1}^{\infty}f_{T_n}(t) \right| \le 5\epsilon,
\end{aligned}
\end{equation*}
and thus we conclude that  for $t>0$
$$
\frac{\partial}{\partial x}f(0^-,t)- \frac{\partial}{\partial x}f(0^+,t)=-\frac{\partial}{\partial x}f(1^-,t).
$$
Similarly, we can get for $t>0$
$$
f(0^-,t)=f(0^+,t).
$$
Now that we have thoroughly checked (\romannumeral6) and  hence, the proof of Theorem \ref{thm_forward} is completed.

\end{proof}

With the same steps as in proving Theorem \ref{thm_forward}, we can show Corollary \ref{thm_forward y}. Next, we only focus on proving Theorem \ref{thm_forward nu}. Due to the results for the process $X_t$ as in \eqref{process definition} that starting from $y<1$ are largely parallel to the one starts from $0$ we have studied in details, only a sketch of proof will be given for those parts.
%\begin{proof}
%With Corollary \ref{thm_forward y}, it suffices to prove the integral relationship \eqref{convolution}. In order to make the integral meaningful, we need to show that for fixed  $x\in(-\infty,1]$, $t\in (0,T]$, the map
%$$
%f^{\cdot}:\ y\in (-\infty,1)\to f^y(x,t).
%$$
%is continuous. Actually, recall the Green function in PDE problem \eqref{forward equation} and we have
%$$
%f^y(x,t)=G(y,0,x,t), \quad x\in(-\infty, 1],t\in(0,T].
%$$
%With Lemma \ref{Green function estimation}, we have
%\beq
%\label{convolution bound}
%\left|\partial^\ell f^y(x,t)\right|\le Ct^{-\frac{1+\ell}{2}}\exp\left(-C_0\frac{(x-y)^2}{t}\right),\quad 0<t\le T.
%\eeq
%where $\ell=0,1,2$, $\partial^{\ell}=\partial_{t x}^{\ell}=\partial_{t}^{m} \partial_{x}^{n}, \,\ell=2 m+n,$ for $m,n \in \mathbb N_0$. Thus we get that the map $f^{\cdot}$ is twice continuously differential and by the property of conditional expectation, for any $t>0$,
%\beq
%\begin{aligned}
%\prob_{\nu}(X_t\le x)&=\e_{\nu}\e_{\nu}[\ind_{X_t\le x}\big|X_0]=\int_{-\infty}^1 \prob_y(X_t\le x)\nu(dy)=\int_{-\infty}^1\int_{-\infty}^{x}G(y,0,z,t)dz\nu(dy),\\
%f^\nu(x,t)&=\frac{d}{dx}\prob_{\nu}(X_t\le x)=\frac{d}{dx}\int_{-\infty}^1\int_{-\infty}^{x}G(y,0,z,t)dz\nu(dy)
%=\int_{-\infty}^1f^y(x,t)\nu(dy).
%\end{aligned}
%\eeq
%In the last equality we exchange the derivative and the integral because by estimation \eqref{convolution bound} $f^y(x,t)$ is bounded.
%\end{proof}

Noting that now $\nu$ is a c.d.f. whose p.d.f. $f_{\text{in}}(x)\in C_c(-\infty,1)$ and that $f_{\text{in}}(x)$ is continuous and compacted supported in $(-\infty, 1-\varepsilon_0]$ for some $\varepsilon_0>0$. Without loss of generality, we assume $f_{\text{in}}(x)$ is supported in $[-C_0, 1-\varepsilon_0]$ for some $C_0>0$. Thus for the fixed $T>0$ we have
\begin{itemize}
    \item[(1)]
    By conditional distribution, we have for any $x\in(-\infty,1]$, $t\in(0,T]$,
    $$
    f^\nu(x,t)=\int_{-\infty}^{1-\varepsilon_0}f^y(x,t)f_{\text{in}}(y)dy.
    $$
    \item[(2)]
    For all $t\in(0,T]$, $x\neq 0 \ \text{or}\ 1$, $f^y(x,t)$ is continuous with respect to $y$.
    \item[(3)]
    All the regularities and convergences in Corollary \ref{thm_forward y} are uniform with respect to $y\in (-\infty,1-\varepsilon_0]$. Actually, for all $\varepsilon_1>0$, $t_0>0$, and any $x\in (-\infty,-\varepsilon_1]\cup[-\varepsilon_1,1)$, $t\in[t_0,T]$, $y\in(-\infty,1-\varepsilon_0]$, we have
    \begin{equation*}
    \begin{aligned}
    \left|f^y(x,t)\right|&\le C_{\varepsilon_0,\varepsilon_1,t_0,T}^{(0)},\quad \left|\partial_xf^y(x,t)\right|\le C_{\varepsilon_0,\varepsilon_1,t_0,T}^{(1)},\\
    \left|\partial_tf^y(x,t)\right|&\le C_{\varepsilon_0,\varepsilon_1,t_0,T}^{(2)},\quad
    \left|\partial_{xx}f^y(x,t)\right|\le C_{\varepsilon_0,\varepsilon_1,t_0,T}^{(3)}.
    \end{aligned}
    \end{equation*}
    Moreover, for all $t\in[t_0,T]$, $y\in(-\infty,1-\varepsilon_0]$ and $x\in[-1,0)\cup(0,1)$
    $$
    \left|\partial_xf^y(x,t)\right|\le C_{\varepsilon_0,t_0,T}.
    $$
    \item[(4)]
    Then we can take the derivative into the integral in \eqref{convolution}, i.e. for $\ell=0,1,2$, $\partial^{\ell}=\partial_{t x}^{\ell}=\partial_{t}^{m} \partial_{x}^{n}, \,\ell=2 m+n$,
    $$
    \partial^\ell f^\nu(x,t)=\int_{-\infty}^{1-\varepsilon_0}\partial^\ell f^y(x,t)\nu(dy), \quad x\in(-\infty,1],\quad t>0,
    $$
    and thus
    $$
    N^\nu(t):=-\partial_xf^\nu(1^-,t)=-\int_{-\infty}^1\partial_x f^y(1^-,t)\nu(dy)= \int_{-\infty}^1N^y(t)\nu(dy).
    $$
    And by the regularities and convergences for $f^y(x,t)$ in Corollary \ref{thm_forward y}, we get the properties (\romannumeral1), (\romannumeral2), (\romannumeral3), (\romannumeral4) and (\romannumeral5) for $f^\nu(x,t)$.
    \item[(5)]
    Finally we check the $L^2$ convergence \eqref{L2 convergence}.
   	We first turn the problem into proving $L^1$ convergence by showing the uniform boundedness of $f^\nu(x,t)$ when $t$ is sufficiently small. In fact, similar to the decomposition in \eqref{density decomposition}, we have
    \beq
    \label{decomposition y}
    f^y(x,t)=\sum_{n=0}^{+\infty}f_n^y(x,t)
    \eeq
    where $f^y_n(x,t)dx=\prob(X^y_t\in dx, n_t=n)$ as in \eqref{iteration_cdf}. With \eqref{decay at infinity}, we have
    \beq
    \label{ou control}
    f_0^y(x,t)\le f_{\text{ou}}^y(x,t)=
    \frac{1}{\sqrt{2\pi(1-e^{-2t})}}\exp{\{\frac{-(x-e^{-t}y)^2}{2(1-e^{-2t})}\}}.
    \eeq
    By the same method in Lemma \ref{CDF iteration lemma}, we get the iteration relationship for any $n\ge 1$
    \beq
    \label{iteration y}
    f_n^y(x,t)=\int_0^tf_{T_1^y}(t-s)f_{n-1}(x,s)ds.
    \eeq
    Using \eqref{density}, we know that for any $t>0$, $f_0(x,t)\le f_{\text{ou}}(x,t)\le \frac{C}{\sqrt{t}}$ and with the similar estimation in Proposition \ref{regularity [0,T]}, we have for any $k\in \BN$, all sufficiently small $t$ and $s\le t$,
    $$
    f_{T^y_1}(t-s)\le C_kt^k,
    $$
    where the constant $C_k$ is independent of all $y\le 1-\varepsilon_0$.
    Thus
    \beq
    \label{iteration calculation}
    f_1^y(x,t)\le C_kt^k\int_0^t\frac{1}{\sqrt{s}}ds\le C_kt^{k+\frac12}.
    \eeq
    Repeat calculations in \eqref{iteration calculation} and with the iteration \eqref{iteration y}, one has for all sufficiently small $t$,
    $$
    f_n^y(x,t)\le (Ct)^n,
    $$
    and thus for all sufficiently small $t$,
    \beq
    \label{summation bound}
    \sum_{n=1}^{+\infty}f_n^y(x,t)\le \frac{Ct}{1-Ct}\le C.
    \eeq
    Combining \eqref{decomposition y}, \eqref{ou control} and \eqref{summation bound}, we have
    \begin{equation*}
    \begin{aligned}
    f^\nu(x,t)\le& \int_{-\infty}^{1-\varepsilon_0}\left[f^y_{\text{ou}}(x,t)
    +C\right]f_{\text{in}}(y)dy\\
    \le&C+\|f_{\text{in}}(y)\|_{L^\infty(-\infty,1-\varepsilon_0]}
    \int_{-\infty}^{1-\varepsilon_0}f^y_{\text{ou}}(x,t)dy.
    \end{aligned}
    \end{equation*}
    Noting that by \eqref{ou control} $\int_{-\infty}^{1-\varepsilon_0}
    f^y_{\text{ou}}(x,t)dy$ is uniformly bounded for any $x$ and sufficiently small $t$, and so does $f^\nu$. Noting that both $f_{\text{in}}(x)$ and $f^\nu(x,t)$ are uniformly bounded for all sufficiently small $t$, thus to prove \eqref{L2 convergence}, it suffices to prove
    \beq
    \label{L1 convergence}
    \lim_{t\to 0^+} \int_{-\infty}^{+\infty}|f^\nu(x,t)-f_{\text{in}}(x)|dx=0.
    \eeq
    To get \eqref{L1 convergence}, for a suitable constant $M_0$ whose value will be specified in the following, we have
    \beq
    \label{L1 decomposition}
    \int_{-\infty}^{+\infty}|f^\nu(x,t)-f_{\text{in}}(x)|dx=
    \left(\int_{-\infty}^{-M_0}+\int_{-M_0}^1\right)|f^\nu(x,t)-f_{\text{in}}(x)|dx
    =:P_1+P_2.
    \eeq
    First to bound $P_1$, we have
    \begin{lemma}
    \label{decay y lemma}
    Now consider the process $X_t$ as in \eqref{process definition} that starts from $y$. For any $\varepsilon>0$, there exists $t_0>0$ and $M_0<\infty$ such that for any $t\in[0,t_0]$ and any $y\in \emph{supp}(f_{\emph{in}})=[-C_0,1-\varepsilon_0]$,
    \beq
    \prob^y(X_t\le -M_0)\le \varepsilon.
    \eeq
    \end{lemma}
    \begin{proof}
    Note that according to the construction of the process $X_t$ as in \eqref{process definition} that starts from $y$, we have
    $$
    \{X_t>-M_0\}\supset\{Y_t^{(1)}>-M_0\}\cap\{T_1>t\}
    $$
    which immediately implies
    \beq
    \label{decay y}
    \prob^y(X_t\le -M_0)\le \prob^y(Y_t^{(1)}\le-M_0)+\prob^y(T_1>t):=Q_1+Q_2.
    \eeq
    For $Q_2$ when $t\le t_0$,
    $$
    \prob^y(T_1\le t)=\int_0^tf_{T^y_1}(s)ds\le C_k\int_0^ts^kds.
    $$
    So let $k=1$ and $t_0=\sqrt{\frac{\varepsilon}{C_1}}$, we have for all $t\le t_0$,
    \beq
    \prob^y(T_1\le t)\le C_1\int_0^tsds\le \frac12 \varepsilon.
    \eeq
    And for $Q_1$, noting that $Y_t^{(1)}$ is Gaussian, we can choose $M_0$ large enough to control $Q_1$ and then complete the proof.

    \end{proof}
    \begin{remark}
    Without loss of generality, we choose the constant $M_0$ in Lemma \ref{decay y lemma} larger than $C_0$.
    \end{remark}
    Lemma \ref{decay y lemma} immediately implies that
    \beq
    \label{decay y 2}
    F^\nu(-M_0,t)=\prob^\nu(X_t\le -M_0)=\int_{-C_0}^{1-\varepsilon_0}\prob^y(X_t\le -M_0)f_{\text{in}}(y)dy\le \varepsilon.
    \eeq
    And for any $\varepsilon>0$, $\exists t_0>0$ and $M_0< \infty$ such that for all $t<t_0$,
    \beq
    \label{L1 convergence 1}
    P_1=\int_{-\infty}^{-M_0}f^\nu(x,t)dx=\prob^\nu(X_t\le -M_0)<\varepsilon.
    \eeq
    To estimate $P_2$ in \eqref{L1 decomposition}, we show in the following
    \begin{lemma}
    \label{density distance}
    For any $\varepsilon>0$, there is a $t_1>0$ such that for any $t\in(0,t_1]$ and $x\in \BR$,
    $$
    f^\nu(x,t)\le f_{\emph{in}}(x)+\varepsilon.
    $$
    \end{lemma}
    \begin{proof}
    Noing that when $x>1$, $f^\nu(x,t)=f_{\text{in}}(x)=0$, thus we only need to focus on $x\in(-\infty,1]$. By \eqref{decomposition y}, \eqref{ou control} and \eqref{summation bound}, we already have
    \beq
    f^\nu(x,t)\le \int_{-\infty}^{+\infty}f^y_{\text{ou}}(x,t)f_{\text{in}}(y)dy+
    \frac{Ct}{1-Ct}
    \eeq
    and $\frac{Ct}{1-Ct}\to 0$ as $t\to 0^+$. Thus we only need to bound $\int_{-\infty}^{+\infty}f^y_{\text{ou}}(x,t)f_{\text{in}}(y)dy$. To do this we will separate the case when $x\in[-C_0-1,1]$ and $x\in(-\infty, -C_0-1)$.
    \begin{itemize}
    \item[(\romannumeral1)]
    When x belongs to the compact set $[-C_0-1,1]$, by \eqref{ou control} we have
    \beq
    \label{normal transformation}
    f^y_{\text{ou}}(x,t)=e^t\frac{1}{\sqrt{2\pi(1-e^{-2t})e^{2t}}}
    \exp\{-\frac{(y-xe^t)^2}{2(1-e^{-2t})e^{2t}}\}
    \eeq
    which equals to the multiply of $e^t$ and the p.d.f. of the normal distribution $N(xe^t,(1-e^{-2t})e^{2t})$. Noting that $f_{\text{in}}(y)$ is uniformly continuous, thus for any $\varepsilon>0$, there $\exists \delta>0$, s.t. for all $|x_1-x_2|\le \delta$, we have $|f_{\text{in}}(x_1)-f_{\text{in}}(x_2)|\le \varepsilon$. And there $\exists t_2>0$ s.t. for all $t<t_2$ and $x\in[-C_0-1, 1]$, $|x-e^tx|<\frac{\delta}{2}$.
    Moreover, for the fixed $\delta$ above, there $\exists t_3>0$ such that for all $t\in (0,t_3)$
    \beq
    \label{Normal bound}
    \prob(|N(0,1)|\ge \frac{\delta}{2\sqrt{(1-e^{-2t})e^{2t}}})\le \frac{\varepsilon}{\|f_{\text{in}}\|_{L^\infty}},
    \eeq
    where $N(0,1)$ stands for the standard normal distribution. Thus for $t_1=t_2\wedge t_3$,
    \beq
    \begin{aligned}
    \int_{-\infty}^{+\infty}f^y_{\text{ou}}(x,t)f_{\text{in}}(y)dy=
    \left(\int_{xe^t-\frac{\delta}{2}}^{xe^t+\frac{\delta}{2}}+
    \int_{\BR\setminus[xe^t-\frac{\delta}{2},xe^t+\frac{\delta}{2}]}\right)
    f^y_{\text{ou}}(x,t)f_{\text{in}}(y)dy=:K_1+K_2
    \end{aligned}
    \eeq
    For $K_1$, we have by \eqref{normal transformation}
    \beq
    \label{K1 bound}
    K_1\le \max_{y\in[x-\delta,x+\delta]}f_{\text{in}}(y)\cdot e^t
    \le \|f_{\text{in}}\|_{L^\infty}(e^t-1)+\max_{y\in[x-\delta,x+\delta]}
    f_{\text{in}}(y)\le f_{\text{in}}(x)+\varepsilon.
    \eeq
    And for $K_2$, we have by \eqref{Normal bound}
    \beq
    \label{K2 bound}
    K_2\le \|f_{\text{in}}\|_{L^\infty}\int_{\BR\setminus[xe^t-\frac{\delta}{2},xe^t+\frac{\delta}{2}]}f^y_{\text{ou}}(x,t)dy
    \le e^t\frac{\varepsilon}{\|f_{\text{in}}\|_{L^\infty}}=e^t\cdot \varepsilon.
    \eeq
    Combing \eqref{K1 bound} and \eqref{K2 bound}, the proof of case (\romannumeral1) is complete.
    \item[(\romannumeral2)]
    Note that $f_{\text{in}}(x)=0$ on $x\in(-\infty, -C_0-1)$ and $f^\nu(x,t)=0$ on $x\ge 1$. We only need to prove that $\forall \varepsilon>0$, $\exists t_1>0$ s.t. $\forall t\in(0,t_1]$ and any $x<-C_0-1$,
    \beq
    \int_{-\infty}^{+\infty}f^y_{\text{ou}}(x,t)f_{\text{in}}(y)dy<\varepsilon.
    \eeq
    By \eqref{ou control} and noting that for any $x<-C_0-1$ and $y\in [-C_0,1]$, we have $|x-e^{-t}y|\ge 1$ and thus
    \begin{equation*}
    \begin{aligned}
    \int_{-\infty}^{+\infty}f^y_{\text{ou}}(x,t)f_{\text{in}}(y)dy\le& (C_0+1)\|f_{\text{in}}\|_{L^\infty}\frac{1}{\sqrt{2\pi(1-e^{-2t})}}
    \exp\left(-\frac{1}{2(1-e^{-2t})}\right)\\
    \le&(C_0+1)\|f_{\text{in}}\|_{L^\infty}\frac{u}{\sqrt{2\pi}}\exp\left(-\frac{u^2}{2}\right),
    \end{aligned}
    \end{equation*}
    where $u:=(1-e^{-2t})^{-\frac12}$. Thus we know $\int_{-\infty}^{+\infty}f^y_{\text{ou}}(x,t)f_{\text{in}}(y)dy\to 0$ as $t\to 0^+$ and the proof of Lemma \ref{density distance} is complete.
    \end{itemize}

    \end{proof}
    With Lemma \ref{density distance}, now we conclude the proof of the \eqref{L2 convergence}. Now for the fixed $M_0$ in Lemma \ref{decay y lemma},  there exists $t_2\ge 0$ s.t. for all $t\in (0,t_2]$ and $x\in \BR$
    $$
    f^\nu(x,t)\le f_{\text{in}}(x)+\frac{\varepsilon}{M_0+1}
    $$
    Noting that $|a-b|\le b-a+2\max\{a-b,0\}$, we have
    \beq
    \begin{aligned}
    \label{L1 convergence 2}
    P_2\le&\int_{-M_0}^1\left[f_{\text{in}}(x) -f^\nu(x,t)+\frac{2\varepsilon}{M_0+1}\right]dx\\
    =&\int_{-M_0}^1f_{\text{in}}(x)dx-\int_{-M_0}^1f^\nu(x,t)dx+2\varepsilon\\
    \le& 3\varepsilon.
    \end{aligned}
    \eeq
    Combining \eqref{L1 convergence 1} and \eqref{L1 convergence 2}, we get \eqref{L2 convergence}. Then the proof of Theorem \ref{thm_forward nu} is complete.

\end{itemize}

\subsection{Weak Solution}
\

\noindent

In this section, we show that the density of $X_t$, which we denote by $f(x,t)$ and $N(t)=\sum_{n=1}^{+\infty} F'_{T_n}(t)$ are the weak solution of the PDE problem \eqref{eq_planck}.
%In the sense of distribution, \eqref{eq_planck} is equivalent to the following
%\begin{equation}
%\left\{
%\begin{aligned}
%\frac{\partial f}{ \partial t} - \frac{\partial }{\partial x}\left ( x f\right) - \frac{\partial ^2 f}{\partial x^2} =\delta(x) N(t), \quad x \in (-\infty, 1),\\
%f(1,t)=0, \quad f(-\infty,t)=0, \quad f(x,0)=\delta(x)
%\end{aligned}
%\right. \label{eq_planck distribution}
%\end{equation}
%where $N(t)=-\frac{\partial}{\partial x}f(1,t).$
We adopt the definition of the weak solution of  \eqref{eq_planck} as in \cite{caceres2011analysis} and the main theorem in this section is as follows:
\begin{theorem}
Let $f^\nu(x,t)$ be the p.d.f of the process $X_t$ as in \eqref{process definition} that starts from p.d.f. $f_{\emph{in}}(x)\in C_c(-\infty,1)$ and $N^\nu(t):=\sum_{n=1}^{+\infty} F'_{T_n}(t)$. The pair $(f, N)$  is a weak solution of \eqref{eq_planck} in the following sense: for any test function $\phi(x, t) \in C^{\infty}\left(\left(-\infty, 1\right] \times[0, T]\right)$ such that $\frac{\partial^{2} \phi}{\partial x^{2}}, x\frac{\partial \phi}{\partial x} \in$
$L^{\infty}\left(\left(-\infty, 1\right] \times[0, T]\right),$ we have
    \beq
	\begin{aligned}
	&\int_0^T \int_{-\infty}^1 \left(\frac{\partial }{\partial t} \phi-x\frac{\partial }{\partial x} \phi+\frac{\partial^2 }{\partial x^2} \phi \right)f^\nu(x,t)dxdt\\
	=&\int_0^T\left(\phi(1, t)-\phi(0, t)\right)N^\nu(t)dt -\int_{-\infty}^1\phi(x,0)f_{\emph{in}}(x)dx+\int_{-\infty}^1 \phi(x, T)f^\nu(x, T)dx
	\end{aligned}
	\eeq
\end{theorem}
The convergence of the series $\sum_{n=1}^{+\infty} F'_{T_n}(t)$ relies on the proof of Theorem \ref{thm_forward nu} with which we have already known that $f^\nu(x,t)$ is a solution to the PDE problem \eqref{eq_planck}. To prove $(f^\nu, N^\nu)$ is also a weak solution of  \eqref{eq_planck}, one simply multiplies the equation by the test function $\phi$ and carries out the integration by parts in space and in time respectively. Since the calculations is rather straightforward, we choose to omit the details in this work but we remark that the weak-strong uniqueness is still an open problem for such a Fokker-Planck equation with a flux-shift structure and we will  continue research along this line in the future.

\section*{Acknowledgement}
J.-G. Liu is partially supported by NSF grant DMS-2106988. Y. Zhang is supported by NSFC Tianyuan Fund for Mathematics 12026606 and the National Key R\&D Program of China, Project Number 2020YFA0712902. Z. Zhou is supported by the National Key R\&D Program of China, Project Number 2020YFA0712000 and NSFC grant No. 11801016, No. 12031013. Z. Zhou is also partially supported by Beijing Academy of Artificial Intelligence (BAAI).

\section*{Appendix}

Now we shall go back to show \eqref{constant<1}. In the following, we let $X_t$ as in \eqref{OU_sol} denote an O-U process starting from $0$ and define the stopping time $T_1$ be the first time that $X_t$ hits $1$, i.e., $T_1=\inf\{t\ge 0, X_t=1\}$. Now it suffices to prove that for all fixed $T\in (0, +\infty)$,
\beq
\label{constant<1 equivalent}
\prob(T_1>T)>0.
\eeq
In order to show \eqref{constant<1 equivalent}, we show the probability of an event included in $\{T_1>T\}$ is positive. Actually, we construct a sequence of stopping time and use the strong Markov property to decompose the process $X_t$ such that each time $|X_t|>1$, it escape from $-1$. By showing the product of the probability of an event sequence is positive, we complete the proof. Now we show a useful lemma.
\begin{lemma}
For the O-U process $X_t$ defined above, define a stopping time $\tau_1=\inf\{t\ge 0, |X_t|=1\}$, then
\begin{numcases}{}
\prob(\tau_1<+\infty)=1 \label{tip1}, \\
\prob(\tau_1>\frac{1}{16},X_{\tau_1}=-1)=\prob(\tau_1>\frac{1}{16},X_{\tau_1}=1)>0. \label{tip2}
\end{numcases}
\end{lemma}
\begin{proof}
\eqref{tip1} follows from the fact that $\tau_1<\inf\{n\in \mathbb N,|X_n|>1\}$, the Markov property and the Gaussian transition distribution of $X_t$. As for \eqref{tip2}, by symmetry, we only need to prove
\beq
\label{tip3}
\prob(\tau_1>\frac{1}{16})>0.
\eeq
By \eqref{OU_sol}, $X_t=\sqrt 2 \int^t_0 e^{-(t-s)}dB_s$ and thus $$
\{\tau_1\le \frac{1}{16}\}=\{\max_{t\le \frac{1}{16}}|X_t|\ge 1\}\subset\{\max_{t\le \frac{1}{16}}\left|\sqrt 2 \int^t_0 e^{-(t-s)}dB_s\right|\ge 1\}.
$$
Now note that $\int^t_0 e^sdB_s$ is a martingale and then
\beq
\begin{aligned}
&\prob(\tau_1\le \frac{1}{16})\le \prob(\max_{t\le \frac{1}{16}}\left|\sqrt 2 \int^t_0 e^{-(t-s)}dB_s\right|\ge 1)\\
\le& 2\e\left(\max_{t\le \frac{1}{16}}\left| \int^t_0 e^sdB_s\right|\right)^2\le 8\e\left(\int^{\frac{1}{16}}_0 e^sdB_s\right)^2,
\end{aligned}
\eeq
where the last two inequalities follows from the Markov inequality and Doob's inequality respectively. Noting that
$$
\e\left(\int_0^{\frac{1}{16}}e^sdB_s\right)^2
=\int_0^{\frac{1}{16}}e^{2s}ds=\frac12(e^{\frac18}-1)<\frac18
$$
and thus \eqref{tip3} is valid.

\end{proof}
With the above lemma, now we prove \eqref{constant<1 equivalent} that is equivalent with \eqref{constant<1}.
\begin{proof}
[Proof of \eqref{constant<1}:]
We let $Y_t$ be an O-U process starting at $-1$ and derive stopping time $\tau'_1=\inf\{t\ge 0, Y_t=0\}$. Then by the recurrence of O-U process,
\beq
\label{tip4}
\prob(\tau'_1<+\infty)=1.
\eeq
Next we define an increasing sequence of stopping times as follows:
\begin{equation*}
\begin{aligned}
S'_0&=0\\
S_1&=\inf\{t\ge 0, |X_t|=1\},\\
S'_1&=\inf\{t\ge S_1, X_t=0\},\\
S_2&=\inf\{t\ge S'_1, |X_t|=1\},\\
S'_2&=\inf\{t\ge S_2, X_t=0\},\\
\vdots
\end{aligned}
\end{equation*}
Combining \eqref{tip1}, \eqref{tip4} and the  Strong Markov Property of the O-U process $S_n, S'_n<+\infty$ for all $n$. At the same time,
$$
S_1-S'_0,\ S'_1-S_1,\  S_2-S'_1,\ \cdots
$$
are independent to each other while
\begin{equation*}
\begin{aligned}
&S_n-S'_{n-1}\overset{d}{=}\tau_1,\\
&S'_n-S_n\overset{d}{=}\tau'_1.
\end{aligned}
\end{equation*}
Thus for the fixed $T\in (0,+\infty)$ above, let $N_0=\lfloor T \rfloor+1$ and then
$$
\{T_1>T\}\supset\cap_{i=1}^{16N_0}\{S_i-S'_{i-1}>\frac{1}{16},X_{S_i}=-1,S'_i-S_i<+\infty\}
$$
Using the strong Markov property, we have
\begin{equation*}
\begin{aligned}
\prob(T_1>T)\ge& \prob\left(\cap_{i=1}^{16N_0}\{S_i-S'_{i-1}>\frac{1}{16},X_{S_i}=-1,S'_i-S_i<+\infty\}\right)\\
=&\prod_{n=1}^{16N_0}\prob(\tau_1>\frac{1}{16}, X_{\tau_1}=-1)>0,
\end{aligned}
\end{equation*}
which completes the proof of \eqref{constant<1}.

\end{proof}

\bibliographystyle{plain}

\bibliography{cite}

\begin{thebibliography}{10}

\bibitem{MR2353702}
Iddo Ben-Ari and Ross~G. Pinsky.
\newblock Spectral analysis of a family of second-order elliptic operators with
  nonlocal boundary condition indexed by a probability measure.
\newblock {\em J. Funct. Anal.}, 251(1):122--140, 2007.

\bibitem{MR2499861}
Iddo Ben-Ari and Ross~G. Pinsky.
\newblock Ergodic behavior of diffusions with random jumps from the boundary.
\newblock {\em Stochastic Process. Appl.}, 119(3):864--881, 2009.

\bibitem{brunel2000dynamics}
Nicolas Brunel.
\newblock Dynamics of sparsely connected networks of excitatory and inhibitory
  spiking neurons.
\newblock {\em Journal of computational neuroscience}, 8(3):183--208, 2000.

\bibitem{brunel1999fast}
Nicolas Brunel and Vincent Hakim.
\newblock Fast global oscillations in networks of integrate-and-fire neurons
  with low firing rates.
\newblock {\em Neural computation}, 11(7):1621--1671, 1999.

\bibitem{caceres2011analysis}
Mar{\'\i}a~J C{\'a}ceres, Jos{\'e}~A Carrillo, and Beno{\^\i}t Perthame.
\newblock Analysis of nonlinear noisy integrate \& fire neuron models: blow-up
  and steady states.
\newblock {\em The Journal of Mathematical Neuroscience}, 1(1):7, 2011.

\bibitem{caceres2011numerical}
Mar{\'\i}a~J C{\'a}ceres, Jos{\'e}~A Carrillo, and Louis Tao.
\newblock A numerical solver for a nonlinear fokker--planck equation
  representation of neuronal network dynamics.
\newblock {\em Journal of Computational Physics}, 230(4):1084--1099, 2011.

\bibitem{Caceres2014}
Mar\'{\i}a~J. C\'{a}ceres and Beno\^{\i}t Perthame.
\newblock Beyond blow-up in excitatory integrate and fire neuronal networks:
  refractory period and spontaneous activity.
\newblock {\em J. Theoret. Biol.}, 350:81--89, 2014.

\bibitem{Caceres2019}
Mar\'{\i}a~J. C\'{a}ceres, Pierre Roux, Delphine Salort, and Ricarda Schneider.
\newblock Global-in-time solutions and qualitative properties for the nnlif
  neuron model with synaptic delay.
\newblock {\em Commun. Part. Diff. Eq.}, 44(12):1358--1386, 2019.

\bibitem{Caceres2018}
Mar\'{\i}a~J. C\'{a}ceres and Ricarda Schneider.
\newblock Analysis and numerical solver for excitatory-inhibitory networks with
  delay and refractory periods.
\newblock {\em ESAIM Math. Model. Numer. Anal.}, 52(5):1733--1761, 2018.

\bibitem{2015Qualitative}
J~Antonio Carrillo, B.~Perthame, D.~Salort, and D.~Smets.
\newblock Qualitative properties of solutions for the noisy integrate and fire
  model in computational neuroscience.
\newblock {\em Nonlinearity}, 28(9), 2015.

\bibitem{carrillo2013classical}
Jos{\'e}~A Carrillo, Mar{\'\i}a d~M Gonz{\'a}lez, Maria~P Gualdani, and Maria~E
  Schonbek.
\newblock Classical solutions for a nonlinear fokker-planck equation arising in
  computational neuroscience.
\newblock {\em Communications in Partial Differential Equations},
  38(3):385--409, 2013.

\bibitem{chevallier2017mean}
Julien Chevallier.
\newblock Mean-field limit of generalized hawkes processes.
\newblock {\em Stochastic Processes and their Applications},
  127(12):3870--3912, 2017.

\bibitem{chevallier2015microscopic}
Julien Chevallier, Mar{\'\i}a~Jos{\'e} C{\'a}ceres, Marie Doumic, and Patricia
  Reynaud-Bouret.
\newblock Microscopic approach of a time elapsed neural model.
\newblock {\em Mathematical Models and Methods in Applied Sciences},
  25(14):2669--2719, 2015.

\bibitem{compte2000synaptic}
Albert Compte, Nicolas Brunel, Patricia~S Goldman-Rakic, and Xiao-Jing Wang.
\newblock Synaptic mechanisms and network dynamics underlying spatial working
  memory in a cortical network model.
\newblock {\em Cerebral cortex}, 10(9):910--923, 2000.

\bibitem{Delarue2015}
F.~Delarue, J.~Inglis, S.~Rubenthaler, and E.~Tanr\'{e}.
\newblock Particle systems with a singular mean-field self-excitation.
  {A}pplication to neuronal networks.
\newblock {\em Stochastic Process. Appl.}, 125(6):2451--2492, 2015.

\bibitem{2019Global}
F.~Delarue, S.~Nadtochiy, and M.~Shkolnikov.
\newblock Global solutions to the supercooled stefan problem with blow-ups:
  regularity and uniqueness.
\newblock 2019.

\bibitem{delarue2013first}
Fran{\c{c}}ois Delarue, James Inglis, Sylvain Rubenthaler, and Etienne
  Tanr{\'e}.
\newblock First hitting times for general non-homogeneous 1d diffusion
  processes: density estimates in small time.
\newblock 2013.

\bibitem{delarue2015global}
Fran{\c{c}}ois Delarue, James Inglis, Sylvain Rubenthaler, Etienne Tanr{\'e},
  et~al.
\newblock Global solvability of a networked integrate-and-fire model of
  mckean--vlasov type.
\newblock {\em The Annals of Applied Probability}, 25(4):2096--2133, 2015.

\bibitem{MR47886}
William Feller.
\newblock The parabolic differential equations and the associated semi-groups
  of transformations.
\newblock {\em Ann. of Math. (2)}, 55:468--519, 1952.

\bibitem{MR63607}
William Feller.
\newblock Diffusion processes in one dimension.
\newblock {\em Trans. Amer. Math. Soc.}, 77:1--31, 1954.

\bibitem{friedman2008partial}
Avner Friedman.
\newblock {\em Partial differential equations of parabolic type}.
\newblock Courier Dover Publications, 2008.

\bibitem{garroni1992green}
Maria~Giovanna Garroni and Jos{\'e}~Luis Menaldi.
\newblock {\em Green functions for second order parabolic integro-differential
  problems}, volume 275.
\newblock Chapman \& Hall/CRC, 1992.

\bibitem{gerstner2002spiking}
Wulfram Gerstner and Werner~M Kistler.
\newblock {\em Spiking neuron models: Single neurons, populations, plasticity}.
\newblock Cambridge university press, 2002.

\bibitem{MR0346904}
\u{I}.~\={I}. G\={\i}hman and A.~V. Skorohod.
\newblock {\em Stochastic differential equations}.
\newblock Springer-Verlag, New York-Heidelberg, 1972.
\newblock Translated from the Russian by Kenneth Wickwire, Ergebnisse der
  Mathematik und ihrer Grenzgebiete, Band 72.

\bibitem{guillamon2004introduction}
T~Guillamon.
\newblock An introduction to the mathematics of neural activity.
\newblock {\em Butl. Soc. Catalana Mat}, 19:25--45, 2004.

\bibitem{2017A}
B.~Hambly and S.~Ledger.
\newblock A stochastic mckean–vlasov equation for absorbing diffusions on the
  half-line.
\newblock {\em The Annals of Applied Probability}, 27(5):2698--2752, 2017.

\bibitem{2018A}
B.~Hambly, S.~Ledger, and A.~Sojmark.
\newblock A mckean--vlasov equation with positive feedback and blow-ups.
\newblock {\em The Annals of Applied Probability}, 2018.

\bibitem{hu2021structure}
Jingwei Hu, Jian-Guo Liu, Yantong Xie, and Zhennan Zhou.
\newblock A structure preserving numerical scheme for fokker-planck equations
  of neuron networks: numerical analysis and exploration.
\newblock {\em Journal of Computational Physics}, 433:110195, 2021.

\bibitem{Inglis2015}
J.~Inglis and D.~Talay.
\newblock Mean-field limit of a stochastic particle system smoothly interacting
  through threshold hitting-times and applications to neural networks with
  dendritic component.
\newblock {\em SIAM J. Math. Anal.}, 47(5):3884--3916, 2015.

\bibitem{lapique1907recherches}
Louis Lapique.
\newblock Recherches quantitatives sur l'excitation electrique des nerfs
  traitee comme une polarization.
\newblock {\em Journal of Physiology and Pathololgy}, 9:620--635, 1907.

\bibitem{mattia2002population}
Maurizio Mattia and Paolo Del~Giudice.
\newblock Population dynamics of interacting spiking neurons.
\newblock {\em Physical Review E}, 66(5):051917, 2002.

\bibitem{2017Particle}
S.~Nadtochiy and M.~Shkolnikov.
\newblock Particle systems with singular interaction through hitting times:
  application in systemic risk modeling.
\newblock {\em Papers}, 2017.

\bibitem{2020Mean}
S.~Nadtochiy and M.~Shkolnikov.
\newblock Mean field systems on networks, with singular interaction through
  hitting times.
\newblock {\em The Annals of Probability}, 48(3):1520--1556, 2020.

\bibitem{newhall2010cascade}
Katherine~A Newhall, Gregor Kova{\v{c}}i{\v{c}}, Peter~R Kramer, and David Cai.
\newblock Cascade-induced synchrony in stochastically driven neuronal networks.
\newblock {\em Physical review E}, 82(4):041903, 2010.

\bibitem{newhall2010dynamics}
Katherine~A Newhall, Gregor Kovacic, Peter~R Kramer, Douglas Zhou, Aaditya~V
  Rangan, David Cai, et~al.
\newblock Dynamics of current-based, poisson driven, integrate-and-fire
  neuronal networks.
\newblock {\em Communications in Mathematical Sciences}, 8(2):541--600, 2010.

\bibitem{Nykamp_thesis}
Duane~Quinn Nykamp.
\newblock {\em A population density approach that facilitates large-scale
  modeling of neural networks}.
\newblock ProQuest LLC, Ann Arbor, MI, 2000.
\newblock Thesis (Ph.D.)--New York University.

\bibitem{omurtag2000simulation}
Ahmet Omurtag, Bruce~W. Knight, and Lawrence Sirovich.
\newblock On the simulation of large populations of neurons.
\newblock {\em Journal of computational neuroscience}, 8(1):51--63, 2000.

\bibitem{MR3244555}
Jun Peng.
\newblock A note on the first passage time of diffusions with holding and
  jumping boundary.
\newblock {\em Statist. Probab. Lett.}, 93:58--64, 2014.

\bibitem{MR3016590}
Jun Peng and WenBo~V. Li.
\newblock Diffusions with holding and jumping boundary.
\newblock {\em Sci. China Math.}, 56(1):161--176, 2013.

\bibitem{MR573203}
Philip Protter.
\newblock Stochastic differential equations with jump reflection at the
  boundary.
\newblock {\em Stochastics}, 3(3):193--201, 1980.

\bibitem{renart2004mean}
Alfonso Renart, Nicolas Brunel, and Xiao-Jing Wang.
\newblock Mean-field theory of irregularly spiking neuronal populations and
  working memory in recurrent cortical networks.
\newblock {\em Computational neuroscience: A comprehensive approach}, pages
  431--490, 2004.

\bibitem{sirovich2006dynamics}
Lawrence Sirovich, Ahmet Omurtag, and Kip Lubliner.
\newblock Dynamics of neural populations: Stability and synchrony.
\newblock {\em Network: Computation in neural systems}, 17(1):3--29, 2006.

\bibitem{MR2673971}
Leszek S\l~omi\'{n}ski and Tomasz Wojciechowski.
\newblock Stochastic differential equations with jump reflection at
  time-dependent barriers.
\newblock {\em Stochastic Process. Appl.}, 120(9):1701--1721, 2010.

\bibitem{Touboul2012}
Jonathan Touboul, Geoffroy Hermann, and Olivier Faugeras.
\newblock Noise-induced behaviors in neural mean field dynamics.
\newblock {\em SIAM J. Appl. Dyn. Syst.}, 11(1):49--81, 2012.

\bibitem{tuckwell1988introduction}
Henry~C Tuckwell.
\newblock {\em Introduction to theoretical neurobiology: volume 2, nonlinear
  and stochastic theories}, volume~8.
\newblock Cambridge University Press, 1988.

\end{thebibliography}
\end{document}